\documentclass[10pt]{amsart}
\usepackage[font=small,labelfont=bf]{caption}
\usepackage[margin=1.25in]{geometry}
\newtheorem{theorem}{Theorem}[section]
\newtheorem{lemma}[theorem]{Lemma}
\newtheorem{proposition}{Proposition}[section]

\newtheorem{corollary}[theorem]{Corollary}
\theoremstyle{definition}
\newtheorem{definition}[theorem]{Definition}

\usepackage{color}
\usepackage{verbatim}
\usepackage{hyperref}
\usepackage[colorinlistoftodos]{todonotes}

\theoremstyle{remark}

\numberwithin{equation}{section}

\usepackage{graphicx, subcaption}
\usepackage{graphicx}
\usepackage{amsmath,amsfonts,amssymb,amsthm}
\usepackage{mathtools}
\usepackage{commath}
\usepackage{colortbl}

\begin{document}
	\title{The Boltzmann-BGK model with nontrivial collision frequency near vacuum}
	
	\author{Sungsu Park,  Seok-Bae Yun}

	\begin{abstract}
        We study the Cauchy problem for the Boltzmann-BGK model in the three-dimensional whole space with nontrivial collision frequency $\nu(f)=\rho^\alpha T^{\beta}$, where $\alpha\in(1/3,1]$, $\beta\in[0,1]$, and $3\alpha+2\beta\geq 2$. Existing well-posedness results for such models have largely been confined to near-equilibrium regimes and stationary problems. For nonnegative initial data with finite mass and energy that are sufficiently small in suitable polynomial-weighted $L^{\infty}$ norms, we establish global existence and uniqueness of mild solutions near vacuum. No uniform positive lower bound on the macroscopic density or temperature is imposed. The analysis relies on two complementary mechanisms. First, a phase-space weight invariant along free-transport characteristics yields dispersive estimates for the collision frequency and the gain term. The resulting time-integrable decay controls the nonlinear growth and closes the global weighted estimates. Second, we establish a Lipschitz estimate for the relaxation operator in a weighted $L^1$ space controlling mass and energy, with a Lipschitz constant depending only on weighted upper bounds for the distribution functions. The relaxation operator $\nu\mathcal{M}$ has a more favorable structure for Lipschitz estimates than the local Maxwellian $\mathcal{M}$ alone: the collision frequency compensates for the singular dependence on the macroscopic fields, allowing us to establish Lipschitz continuity without uniform positive lower bounds for the density or temperature. Together, these estimates yield a unique global mild solution with uniform weighted bounds.
	\end{abstract}
	\maketitle
\tableofcontents
\section{Introduction}
The Boltzmann equation provides a kinetic description of monatomic, non-ionized gases. However, the complexity of its collision operator poses a significant challenge to the development of efficient numerical methods. To address this difficulty, Bhatnagar, Gross, and Krook \cite{bhatnagar1954model} proposed a simplified kinetic model, which is called the BGK model, where the collision operator is replaced by a relaxation term. In this paper, we consider the Cauchy problem for the Boltzmann-BGK model, which describes the phase space evolution of gases:
\begin{align}\label{BGK model}
    \begin{split}
        \partial_t f +v\cdot \nabla_x f  &=\nu_{\alpha,\beta}(f)( \mathcal{M}(f)-f),\cr        
        f(0)&=f_0.
    \end{split}
\end{align}
The velocity distribution function $f=f(t,x,v)$ represents the number density at the phase space point $(x,v)\in\mathbb{R}^{3}\times\mathbb{R}^{3}$ at time $t\in\mathbb{R}^{+}$. 
The local Maxwellian $\mathcal{M}(f)$ reads
\begin{align}\label{Maxwellian def}
    \begin{split}
        \mathcal{M}(f) = \frac{\rho_f}{(2\pi T_f)^{\frac{3}{2}}}\exp\left(-\frac{|v-U_f|^2}{2T_f} \right),
    \end{split}
\end{align}
with the macroscopic local density $\rho(t,x)$, momentum $U(t,x)$, and temperature $T(t,x)$ which defined by
\begin{align}\label{Macroscopic fields}
    \begin{split}
       \rho_f(t,x) &= \int_{\mathbb{R}^{3}} f(t,x,v)dv, \cr
       \rho_f(t,x) U_f(t,x) &= \int_{\mathbb{R}^{3}} vf(t,x,v)dv, \cr
       \rho_f(t,x)(|U_f(t,x)|^{2} &+ 3T_f(t,x))= \int_{\mathbb{R}^{3}} |v|^{2}f(t,x,v) dv.
    \end{split}
\end{align}
The collision frequency $\nu_{\alpha,\beta}(f)$ is given by
\begin{align*}
    \begin{split}
        \nu_{\alpha,\beta}(f)=\rho^\alpha_f T_f^{\beta},\text{ for }\alpha,\beta\in[0,1].
    \end{split}
\end{align*}
The relaxation operator $\mathcal{M}(f)$ preserves the important properties of the Boltzmann equation. First, since $\mathcal{M}(f)$ shares the first three moments with $f$, the relaxation operator satisfies 
\begin{align*}
    \begin{split}
        \int_{\mathbb{R}^{3}}(\mathcal{M}(f)-f)(1,v,|v|^2)dv = 0,
    \end{split}
\end{align*}
which leads to the conservation of mass, momentum, and energy:
\begin{align*}
    \begin{split}
        \frac{d}{dt}\int_{\mathbb{R}^{6}}(1,v,|v|^2)f dxdv = 0.
    \end{split}
\end{align*}
On the other hand, the BGK model satisfies the entropy dissipation:
\begin{align*}
    \begin{split}
        \int_{\mathbb{R}^{3}}(\mathcal{M}(f)-f)\ln fdv \leq 0,
    \end{split}
\end{align*}
which gives the H-theorem:
\begin{align*}
    \begin{split}
        \frac{d}{dt}\int_{\mathbb{R}^{3}}f\log f dxdv \leq 0.
    \end{split}
\end{align*}

\subsection{Main result} Before presenting the main result, we provide notation that will be frequently used.
\begin{itemize}
    \item $C_{a,b,c,\cdots}$ means a constant which depends on $a$, $b$, $c$, $\cdots$. Moreover, these constants may take different values in each line. If no subscript is specified, like $C$, this represents a generic constant.
    \item We use $\rho$, $U$, $T$, and $\rho_f$, $U_f$, $T_f$ interchangeably to denote the macroscopic fields constructed from $f$. We use the latter case when there is a risk of confusion.     
    \item We define the weight functions and their corresponding norms as follows: 
    \begin{align*}
        \begin{split}
            &w_a(v)=(1+|v|^2)^{\frac{a}{2}},\quad W_a(t,x,v)=(1+|x-vt|^2+|v|^2)^{\frac{a}{2}},\cr
            &\|f(t)\|_{w_a}=\sup_{x,v\in\mathbb{R}^3}|w_{a}(v)f(t,x,v)|,\cr
            &\|f(t)\|_{W_a}=\sup_{x,v\in\mathbb{R}^3}|W_{a}(t,x,v)f(t,x,v)|.
        \end{split}
    \end{align*}
    \item We define the weighted $L^1$ norm as follows:
    \begin{align*}
        \begin{split}
            \|f(t)\|_{L^1_2(\mathbb{R}^6)}=\int_{\mathbb{R}^6}(1+|v|^2)|f(t,x,v)|dxdv.
        \end{split}
    \end{align*}
\end{itemize}
We now define our solution space and the mild solution to \eqref{BGK model}.
\begin{definition}\label{def:Sol def}
    The distribution function $f\in C([0,\infty);L^1_2)$ is said to be a mild solution to \eqref{BGK model} if it satisfies
    \begin{align*}
        \begin{split}
            f(t,x,v)=f_0(x-vt,v)+\int^t_0\nu_{\alpha,\beta}(f)(s,x-v(t-s))(\mathcal{M}(f)-f)(s,x-v(t-s),v)ds,
        \end{split}
    \end{align*}
    for almost every $(t,x,v)\in[0,\infty)\times\mathbb{R}^3\times \mathbb{R}^3$.
\end{definition}

We are now ready to state our main result.
\begin{theorem}\label{thm:main thm}
    Let
    \begin{align*}
        \begin{split}
            &\frac{1}{3}<\alpha\leq 1,\quad 0\leq \beta\leq 1,\quad 3\alpha+2\beta\geq 2,\cr
            &q>\max\left(3+\frac{2(1+\beta)}{\alpha},3+\frac{6\beta}{3\alpha-1} \right),\cr
            &r\in \left(3,3\alpha+2-\frac{3(3\alpha+2\beta-1)}{q} \right)\cup(5,\infty).
        \end{split}
    \end{align*}
    Assume that $f_0\geq0$, $\|f_0\|_{L^1_2}<\infty$, and there exists a sufficiently small $\varepsilon>0$ depending only on $\alpha$, $\beta$, $r$, and $q$ such that 
    \begin{align*}
        \begin{split}
            \|f_0\|_{w_q},\,\|f_0\|_{W_r}\leq \varepsilon^2.
        \end{split}
    \end{align*}
    Then, there is a unique mild solution to \eqref{BGK model} which satisfies
    \begin{itemize}
        \item $f$ is nonnegative.
        \item $f$ has the following upper bound:
        \begin{align*}
            \begin{split}
                \sup_t \|f(t)\|_{w_q}\leq \varepsilon,\quad \sup_t\|f(t)\|_{W_r}\leq \varepsilon.
            \end{split}
        \end{align*}
        \item $\nu_{\alpha,\beta}(f)$ has the following upper bound:
        \begin{align*}
            \begin{split}
                \nu_{\alpha,\beta}(f)(t,x)\leq C_{r,q}\varepsilon^\alpha(1+t^2)^{-\frac{3\alpha}{2}+\frac{3\beta}{q-3}}.
            \end{split}
        \end{align*}
        \item $f$ is uniformly bounded in $L^1_2$:
        \begin{align*}
            \begin{split}
                \sup_t\|f(t)\|_{L^1_2}\leq 2\|f_0\|_{L^1_2}.
            \end{split}
        \end{align*}
    \end{itemize}
\end{theorem}

\subsection{Key challenges}\label{S1-1}
The collision frequency in the BGK model generally depends on the local macroscopic density and temperature, which is expressed in the following form:
\begin{align*}
    \begin{split}
        \nu_{\alpha,\beta}=\rho^{\alpha}T^{\beta},\text{ for }\alpha,\beta\in[0,1].
    \end{split}
\end{align*}
The selection of $\alpha$ and $\beta$ reflects the type of gas under consideration \cite{itikawa1973effective}. For instance, Argon assumes the value $\alpha=1$ and $\beta=0.19$ \cite{chapman1970mathematical}. In \cite{bhatnagar1954model,mischler1996uniqueness,perthame1989global,perthame1993weighted,russo2009semilagrangian,russo2012convergence,russo2018convergence}, authors consider a constant collision frequency, corresponding to $\alpha=\beta= 0$. Several studies also address nontrivial collision frequencies, for which at least one of these exponents is nonzero. In particular, \cite{aoki1990numerical,boscheri2020high,pieraccini2007implicit} consider the density-dependent case $\alpha=1$ and $\beta=0$, while \cite{chapman1970mathematical,mieussens2000discrete2,rapp2010equilibration,yang1995rarefied,yang2013kinetic} treat a collision frequency depending on both density and temperature, with $\alpha,\beta\in(0,1]$. Although the well-posedness of BGK-type models with constant collision frequencies has been extensively studied, the corresponding theory for nontrivial collision frequencies remains less developed. Despite the numerical and analytical studies discussed above, rigorous well-posedness results have largely been restricted to near-equilibrium or stationary frameworks.\\
\indent Motivated by this gap, we establish global well-posedness near vacuum for nontrivial collision frequencies with a wide range of $\alpha$ and $\beta$ which satisfies the hypothesis of Theorem~\ref{thm:main thm}. This setting encompasses both density-dependent and density- and temperature-dependent collision frequencies considered in the aforementioned studies. Two principal difficulties arise in the analysis. First, the dependence of the collision frequency on the macroscopic fields introduces nonlinear growth of the distribution function. Second, the near-vacuum regime does not provide uniform positive lower bounds for the density and temperature, thereby obstructing the standard Lipschitz estimates for the local Maxwellian. The core of our analysis rests on two properties established in Lemma~\ref{lem:density and Maxwellian est} and Proposition~\ref{prop:Lip prop}: the dispersive effect of the nontrivial collision frequency in the whole space and a Lipschitz estimate that exploits the more favorable structure of the full gain operator $\nu_{\alpha,\beta}\mathcal{M}(f)$ compared with the Maxwellian $\mathcal{M}(f)$ alone.\\
$\bullet$ \textbf{Dispersive effect for the collision frequency in $\mathbb{R}^3$.} To explain the first challenge, we recall the standard Maxwellian estimate in a velocity-weighted $L^{\infty}$ space with a sufficiently large velocity exponent $q>5$ in \cite{perthame1993weighted}:
\begin{align*}
    \begin{split}
        \|\mathcal{M}(f)\|_{w_q}\leq C_q\|f\|_{w_q}.
    \end{split}
\end{align*}
When the collision frequency is constant, this estimate yields a linear Gronwall's inequality for the weighted $L^{\infty}$ norm of the distribution function.
\begin{align*}
    \begin{split}
        \|f(t)\|_{w_q} \leq e^{-t}\|f_0\|_{w_q}+C_q\int^{t}_0 e^{-(t-s)}\|f(s)\|_{w_q}ds.   
    \end{split}
\end{align*}
However, for a solution-dependent collision frequency, applying the same Maxwellian estimate directly to \eqref{BGK model} leads to a nonlinear growth estimate:
\begin{align*}
    \begin{split}
        \|f(t)\|_{w_q} \leq \|f_0\|_{w_q}+C_{q,\alpha}\int^{t}_0 \|f(s)\|^{1+\alpha}_{w_q}ds.    
    \end{split}
\end{align*}
This inequality does not provide an effective global-in-time upper bound. The obstruction lies in the additional dependence of the collision frequency on the size of the distribution function, which prevents the preceding linear Gronwall's argument from closing. To overcome this issue, we introduce a free-transport weight $W_r=(1+|x-vt|^2+|v|^2)^{\frac{r}{2}}$ and the associated weighted $L^{\infty}$ norm $\|f(t)\|_{W_r}=\sup_{x,v}|W_r(t,x,v)f(t,x,v)|$. This weight is constant along free-transport characteristics. Consequently, when taking the weight to the Lagrangian-type mild formulation, its value at the observation point agrees with its value at the corresponding phase-points in the initial and source terms:
\begin{align*}
    \begin{split}
        W_r(t,x,v)=W_r(0,x-vt,v)=W_r(s,x-v(t-s),v),\text{ for }s\in[0,t].
    \end{split}
\end{align*}
Control in this norm allows us to exploit the dispersive effect in $\mathbb{R}^3$ and derive a decay estimate for the collision frequency (see \eqref{lem d and M nu} in  Lemma~\ref{lem:density and Maxwellian est}):
\begin{align*}
    \begin{split}
        \nu_{\alpha,\beta}(f)\leq C_{r,q}\|f(t)\|^{\alpha-\frac{2\beta}{q-3}}_{W_r}\|f(t)\|^{\frac{2\beta}{q-3}}_{w_q}(1+t^2)^{-\frac{3\alpha}{2}+\frac{3\beta}{q-3}},
    \end{split}
\end{align*}
Together with the weighted Maxwellian estimates, this dispersive mechanism yields time decay for the gain term in both the velocity-weighted and the free-transport-weighted $L^{\infty}$ norms (see \eqref{lem d and M M1} and \eqref{lem d and M M2} in Lemma~\ref{lem:density and Maxwellian est}):
\begin{align*}
    \begin{split}
        &\|\nu_{\alpha,\beta}(f)\mathcal{M}(f)(t)\|_{w_q}\leq C_{r,q}\|f(t)\|^{\alpha-\frac{2\beta}{q-3}}_{W_r}\|f(t)\|^{1+\frac{2\beta}{q-3}}_{w_q}(1+t^2)^{-\frac{3\alpha}{2}+\frac{3\beta}{q-3}},\cr
        &\|\nu_{\alpha,\beta}\mathcal{M}(f)(t)\|_{W_r}\leq C_{r,q,\beta}(\|f(t)\|^{1+\alpha}_{W_r}+\|f(t)\|^{1+\alpha}_{w_q})(1+t^2)^{-\frac{\gamma}{2}}.
    \end{split}
\end{align*}
Combining these estimates with the smallness assumption on the initial data $\|f_0\|_{w,q},\|f_0\|_{W_r}\leq \varepsilon^2$, we improve the preceding nonlinear growth estimate as follows:
\begin{align*}
    \begin{split}
        &\|f(t)\|_{w_q}\leq \varepsilon^2+C_{q,\alpha}\varepsilon^{\alpha+1}\int^t_0 (1+s)^{-\frac{3\alpha}{2}+\frac{3\beta}{q-3}}ds,\cr
        &\|f(t)\|_{W_r}\leq \varepsilon^2+C_{r,q,\alpha,\beta}\varepsilon^{\alpha+1}\int^t_0 (1+s)^{-\gamma}ds,\text{ for }\gamma>1.
    \end{split}
\end{align*}
The crucial feature of this improvement is the integrability of the resulting time-decay factor over the entire positive time axis. Therefore, its contribution can be controlled uniformly in the final time. This allows the smallness threshold $\varepsilon$ for the initial data to be chosen independently of the final time and closes the global a priori estimate.\\
$\bullet$ \textbf{Lipschitz continuity through compensation by the collision frequency.} We next describe the Lipschitz estimate for the gain operator (see Proposition~\ref{prop:Lip prop}):
\begin{align*}
    \begin{split}
        \|\nu_{\alpha,\beta}(f)\mathcal{M}(f)-\nu_{\alpha,\beta}(g)\mathcal{M}(g)\|_{L^1_2}\leq C_{Lip}\varepsilon^\alpha\|f-g\|_{L^1_2(\mathbb{R}^6)}.
    \end{split}
\end{align*}
The key advantage of this estimate is that it does not require uniform positive lower bounds for the density or temperature. By contrast, in the standard analysis of the BGK model with constant collision frequency, Lipschitz estimates for the Maxwellian typically involve Lipschitz constants depending on positive lower bounds for these macroscopic quantities: 
\begin{align*}
    \begin{split}
        &|\mathcal{M}(f)-\mathcal{M}(g)|\cr
        &\leq |\rho_f-\rho_g|\int^1_0\left|\frac{\partial \mathcal{M}_h}{\partial \rho_h} \right|d\theta+|U_f-U_g|\int^1_0\left|\frac{\partial \mathcal{M}_h}{\partial U_h} \right|d\theta+|T_f-T_g|\int^1_0\left|\frac{\partial \mathcal{M}_h}{\partial T_h} \right|d\theta\cr
        &\leq \left(1+|U_h|^2+T_h+\frac{1}{T_h}+\frac{|U_h|^3}{T_h} \right)|(|\rho_f-\rho_g|+|U_f-U_g|+|T_f-T_g|).
    \end{split}
\end{align*}
Accordingly, suitable assumptions for the lower bound are often imposed on the initial data to obtain the required macroscopic bounds. Such assumptions, however, are not natural in the near-vacuum regime considered here. Our Lipschitz estimate is instead compatible with this regime, as its Lipschitz constant is controlled by only a weighted upper bound for the distribution function via Lemma~\ref{lem:Perthame ineq}:
\begin{align*}
    \begin{split}
        C_{Lip} &\sim C_{\alpha}\{\rho^\alpha_{h}T^{\beta}_{h}(1+|U_{h}|^2+T_{h})+ \rho^\alpha_{h}T^{\beta-1}_{h}+\rho^\alpha_{h}|U_{h}|^3T^{\beta-1}_{h}+\rho^\alpha_{h}|U_{h}|^4T^{\beta-1}_{h}\}\cr
        &\leq C_{q,\alpha}(\theta\|f\|_{w_q}+(1-\theta)\|g\|_{w_q})^\alpha\leq C_{q,\alpha}\varepsilon^\alpha,\text{ for }\theta\in[0,1]. 
    \end{split}
\end{align*}
Moreover, these Lipschitz constants are adapted from the derivatives of the gain and loss operators, which are given by $\nu_{\alpha,\beta}(f)\mathcal{M}(f)$ and $\nu_{\alpha,\beta}(f)f$ respectively, with respect to the macroscopic variables $(\rho_f,\rho_fU_f,\rho_f(|U_f|^2+3T_f)$. These bounds yield an $L^1_2$ stability estimate for any two solutions $f$ and $g$ in the solution class, constructed in Proposition~\ref{prop:cauchy prop}. Applying the same estimate to the Picard-type iteration introduced in \eqref{appr system}, we show that the approximate solutions form a Cauchy sequence on every finite time interval. Passing to the limit in the nonlinear terms then produces a global mild solution. The stability estimate also gives uniqueness and thereby establishes well-posedness for the BGK model with nontrivial collision frequency. To summarize, the global upper bounds obtained through dispersion and the stability estimates on finite time intervals thus play distinct, complementary roles. The decay estimate controls the accumulated effect of the collision frequency, while the structure of this frequency within the gain term makes stability estimates possible without uniform positive lower bounds for the density or temperature. Together, these two mechanisms complete the global well-posedness argument near vacuum. 

\subsection{Literature review}
We briefly review the well-posedness theory for BGK-type models. For constant collision frequencies, Perthame \cite{perthame1989global} established the existence of global weak solutions, and Perthame and Pulvirenti \cite{perthame1993weighted} developed weighted estimates and proved uniqueness. The uniqueness theory was extended to the whole space in \cite{mischler1996uniqueness}. Further developments include regularity results for both BGK and ES-BGK models: \cite{issautier1996convergence,yun2015classical}. The $L^p$ problem was studied in \cite{zhang2007lp}. In \cite{chen2018smooth,chen2016global,son2023bgk}, the infinite energy setting for the various BGK-type models was investigated. Incorporating mean-field interactions has also been studied in \cite{wei2012cauchy,zhang2010cauchy}. Related results concern the near-equilibrium regime \cite{bellouquid2003global} and the stationary problems \cite{brull2020stationary}.\\
\indent For nontrivial collision frequencies, the available well-posedness theory is more limited. As discussed in Section~\ref{S1-1}, existing results largely focused on perturbations of equilibrium and stationary configurations. In particular, near-equilibrium problems have been investigated in \cite{yun2010cauchy,yun2019ellipsoidal} for the BGK, ES-BGK, and polyatomic ES-BGK models. Indeed, the perturbation theories for initial data with large amplitude and large oscillations were established in \cite{bae2026large} and \cite{duan2017global}. The existence and uniqueness of a stationary solution have been studied in \cite{bang2016stationary}. Yun \cite{Yun2026} recently established the existence of global renormalized solutions to the BGK model with a physical collision frequency, assuming that the initial data have finite mass, momentum, energy, and entropy.\\
\indent Moreover, we offer the results in the near-vacuum regime for the Boltzmann equation. The near-vacuum regime roughly refers to
a solution class in which the solution decays both in space and velocity, and the dispersive effect prevails in the long run. A typical example is the traveling local Maxwellian. The first such results can be traced back to \cite{illner1984boltzmann}, in which
Illner and Shinbrot proved the existence and uniqueness of global-in-time solutions for the Boltzmann equation in the hard-sphere case when the initial data is bounded by a sufficiently small traveling local Maxwellian. Extensions to the polynomial(in $x$)-exponential(in $v$) mixed decay, and purely polynomial decay in both $x$ and $v$  were established respectively in \cite{bellomo1985cauchy, toscani1986non}. Toscani provided a rigorous proof of the $H$-theorem in \cite{toscani1987h} even when dispersive asymptotics prevail over thermal equilibrization in the near-vacuum regime. In \cite{toscani1988global}, the well-posedness of the Boltzmann equation was obtained when the initial data lie sufficiently close to a traveling local Maxwellian for the soft potential. The hard potential case was treated in \cite{palczewski1989global}. Ha \cite{ha20041, ha2005nonlinear} obtained the $L^1$-stability of the Boltzmann equation near vacuum. Glassey \cite{glassey2006global} identified the relativistic version of the near-vacuum regime and established the existence of the relativistic Boltzmann equation. Guo \cite{guo2001vlasov} studied global classical solutions of the Vlasov-Poisson-Boltzmann system near vacuum for soft potentials. Extension to the polynomial decay in both $x$ and $v$ was studied by Duan, Yang, and Zhu \cite{duan2006boltzmann}. In \cite{duan2006l1}, $L^1$-stability results of the Vlasov-Poisson-Boltzmann system near vacuum were considered. Chae, Ha, and Hwang \cite{chae2006time} introduced a collision potential to study the time-asymptotic behavior.
\newline
\newline
\indent This paper is organized as follows. In Section~\ref{S2}, we present a dispersive estimate for the collision frequency and the gain term of the BGK operator. In Section~\ref{S3}, we define the solution space and construct a sequence of approximate solutions through a Picard-type iteration. We then show that every iterate belongs to this class. In Section~\ref{S4}, we provide the Lipschitz continuity of the gain term and use it to prove that the sequence constructed in Section~\ref{S3} is Cauchy. Finally, Section~\ref{S5} is devoted to the proof of our main theorem.

\section{A priori estimates}\label{S2}
In this section, we introduce the technical lemma, which controls some specific structure of macroscopic fields using the distribution function. 
\begin{lemma}\label{lem:Perthame ineq}
    Let $\alpha\in(0,1]$ and $\beta\in[0,1]$, and assume that
    \begin{align*}
        \begin{split}
            3\alpha+2\beta\geq 2,\quad q>3+\frac{2(1+\beta)}{\alpha}.
        \end{split}
    \end{align*}
    Then, for nonnegative $f$, the following estimates hold:
    \begin{align*}
        \begin{split}
            &(1)\; \rho_f T^{-\frac{3}{2}}_f\leq C_q\|f\|_{w_q},\cr
            &(2)\; \rho_f (|U_f|^2+T_f)^{\frac{q-3}{2}}\leq C_q\|f\|_{w_q},\cr
            &(3)\; \frac{\rho_f|U_f|^{q+3}}{\{T_f(T_f+|U_f|^2)\}^{\frac{3}{2}}} \leq C_q\|f\|_{w_q}.
        \end{split}
    \end{align*}
    Moreover, we have
    \begin{align*}
        \begin{split}
            &(4)\; \rho^\alpha_fT^{\beta-1}_f\leq C_{q,\alpha} \|f\|^\alpha_{w_q},\cr
            &(5)\; \rho^\alpha_f|U_f|^mT^{\beta-1}_f\leq C_{q,\alpha} \|f\|^\alpha_{w_q},\text{ for }2<m<\alpha(q-3)+2-2\beta,\cr
            &(6)\;\rho^\alpha_fT^{\beta}_f(1+|U_f|^2+T_f)\leq C_{q,\alpha}\|f\|^\alpha_{w_q}.
        \end{split}
    \end{align*}    
\end{lemma}
\begin{proof}
    The first three assertions have already been proven in \cite{perthame1993weighted}. Thus, we will only show estimates $(4)-(6)$.\\
    $\bullet$ \textbf{Estimate $(4)$:} For $T_f\geq 1$, we have
    \begin{align*}
        \begin{split}
            \rho^\alpha_fT^{\beta-1}_f\leq \rho^\alpha_f\leq  \left(\|f\|_{w_q}\int_{\mathbb{R}^3}\frac{dv}{(1+|v|^2)^{\frac{q}{2}}}\right)^{\alpha}\leq C_{q,\alpha} \|f\|^\alpha_{w_q}.
        \end{split}
    \end{align*}
    For $T_f\leq 1$, estimate $(1)$ gives
    \begin{align*}
        \begin{split}
            \rho^\alpha_fT^{\beta-1}_f=\left(\rho_fT^{-\frac{3}{2}}_f\right)^{\alpha}T^{\frac{3\alpha}{2}+\beta-1}_f\leq C_{q,\alpha}\|f\|^\alpha_{w_q}T^{\frac{3\alpha}{2}+\beta-1}_f\leq C_{q,\alpha}\|f\|^\alpha_{w_q},
        \end{split}
    \end{align*}
    where we used $3\alpha+2\beta\geq 2$.\\
    $\bullet$ \textbf{Estimate $(5)$:} We divide the cases into two parts such that $(i)\;|U_f|^2\leq T_f$, and $(ii)\; |U_f|^2>T_f$. In the former case $(i)$, we have
    \begin{align*}
        \begin{split}
            \rho^\alpha_f|U_f|^m T^{\beta-1}_f\leq \rho^\alpha_fT^{\frac{m}{2}+\beta-1}_f.
        \end{split}
    \end{align*}
    Since $2<m<\alpha(q-3)+2-2\beta$, we see 
    \begin{align*}
        \begin{split}
            0<\frac{m}{2}+\beta-1<\frac{\alpha(q-3)}{2}.
        \end{split}
    \end{align*}
    Thus, estimate $(2)$ gives
    \begin{align*}
        \begin{split}
           &\rho^\alpha_fT^{\frac{m}{2}+\beta-1}_f\leq \rho^\alpha_f\leq C_{q,\alpha}\|f\|^\alpha_{w_q}\quad(T_f\leq 1),\cr
           &\rho^\alpha_fT^{\frac{m}{2}+\beta-1}_f\leq \left(\rho_fT^{\frac{q-3}{2}}_f\right)^{\alpha}\leq C_{q,\alpha}\|f\|^\alpha_{w_q}\quad(T_f> 1).
        \end{split}
    \end{align*}
    In the latter case $(ii)$, we observe
    \begin{align*}
        \begin{split}
            \rho^\alpha_f|U_f|^mT^{\beta-1}_f=\left(\frac{\rho_f}{T^{\frac{3}{2}}_f}\right)^{\alpha} |U_f|^{m}T^{\frac{3\alpha}{2}+\beta-1}_f\leq \left(\frac{\rho_f}{T^{\frac{3}{2}}_f}\right)^{\alpha} |U_f|^{m+3\alpha+2\beta-2}.     
        \end{split}
    \end{align*}
    Interpolating estimates (1) and (3), we get
    \begin{align*}
        \begin{split}
           \left(\frac{\rho_f}{T^{\frac{3}{2}}_f}\right)^{\alpha} |U_f|^{m+3\alpha+2\beta-2}\leq \left(\frac{\rho_f}{T^{\frac{3}{2}}_f} \right)^{\alpha-\frac{m+3\alpha+2\beta-2}{q}}\left(\frac{\rho_f|U_f|^q}{T^{\frac{3}{2}}_f} \right)^{\frac{m+3\alpha+2\beta-2}{q}}\leq C_{q,\alpha}\|f\|^\alpha_{w_q},   
        \end{split}
    \end{align*}
    where we used $0<\frac{m+3\alpha+2\beta-2}{q}<\alpha$. This proves $(5)$.\\
    $\bullet$ \textbf{Estimate (6): }Since $q>3+\frac{2(1+\beta)}{\alpha}$, we have
    \begin{align*}
        \begin{split}
            1+\beta<\frac{\alpha(q-3)}{2}.
        \end{split}
    \end{align*}
    Therefore, we observe
    \begin{align*}
        \begin{split}
            \rho^\alpha_fT^{\beta}_f(1+|U_f|^2+T_f)&\leq \rho^\alpha_f(1+|U_f|^2+T_f)^{1+\beta}\cr
            &\leq \rho^\alpha_f(1+|U_f|^2+T_f)^{\frac{\alpha(q-3)}{2}}\cr
            &\leq C_{q,\alpha}\left\{\rho^\alpha_f+(\rho_f(|U_f|^2+T_f)^{\frac{q-3}{2}})^{\alpha}\right\}.
        \end{split}
    \end{align*}
    Then, $\rho_f\leq C_q\|f\|_{w_q}$ and estimate $(2)$ give the desired result.
\end{proof}

\begin{corollary}\cite{perthame1993weighted}\label{cor:Mx lemma}
    For $q>5$, we have
    \begin{align*}
        \begin{split}
            \|\mathcal{M}(f)(t)\|_{w_q}\leq C_{q}\|f(t)\|_{w_q}.
        \end{split}
    \end{align*}
\end{corollary}

In the following lemma, we establish the dispersive estimate for the collision frequency $\nu_{\alpha,\beta}(f)$ and the gain operator $\nu_{\alpha,\beta}(f)\mathcal{M}(f)$. 
\begin{lemma}\label{lem:density and Maxwellian est}
    Let $\alpha\in(1/3,1]$ and $\beta\in[0,1]$, and assume that
    \begin{align*}
        \begin{split}
            q>\max\left(3+\frac{2(1+\beta)}{\alpha},3+\frac{6\beta}{3\alpha-1} \right),\quad r>3.
        \end{split}
    \end{align*}
    Then, for nonnegative $f$, we have
    \begin{align}\label{lem d and M d}
        \begin{split}
            \rho_f(t,x)\leq C_r\|f(t)\|_{W_r}(1+t^2)^{-\frac{3}{2}}.   
        \end{split}
    \end{align}
    \begin{align}\label{lem d and M nu}
        \begin{split}
            \nu_{\alpha,\beta}(f)\leq C_{r,q}\|f(t)\|^{\alpha-\frac{2\beta}{q-3}}_{W_r}\|f(t)\|^{\frac{2\beta}{q-3}}_{w_q}(1+t^2)^{-\frac{3\alpha}{2}+\frac{3\beta}{q-3}}.
        \end{split}
    \end{align}
    \begin{align}\label{lem d and M M1}
        \begin{split}
            \|\nu_{\alpha,\beta}(f)\mathcal{M}(f)(t)\|_{w_q}\leq C_{r,q}\|f(t)\|^{\alpha-\frac{2\beta}{q-3}}_{W_r}\|f(t)\|^{1+\frac{2\beta}{q-3}}_{w_q}(1+t^2)^{-\frac{3\alpha}{2}+\frac{3\beta}{q-3}}.
        \end{split}
    \end{align}
    Moreover, for $r>5$, we have
    \begin{align}\label{lem d and M M1-1}
        \begin{split}
            \|\nu_{\alpha,\beta}(f)\mathcal{M}(f)(t)\|_{W_r}\leq C_{r,q}\|f(t)\|^{\alpha+1-\frac{2\beta}{q-3}}_{W_r}\|f(t)\|^{\frac{2\beta}{q-3}}_{w_q}(1+t^2)^{-\frac{3\alpha}{2}+\frac{3\beta}{q-3}}.  
        \end{split}
    \end{align}
    For $3<r<3\alpha+2-\frac{3(3\alpha+2\beta-1)}{q}$, define 
    \begin{align*}
        \begin{split}
           \gamma=3(\alpha+1)-r-\frac{3(2\beta+r-3)}{q-3}>1.
        \end{split}
    \end{align*}
    Then, we have
    \begin{align}\label{lem d and M M2}
        \begin{split}
            \|\nu_{\alpha,\beta}(f)\mathcal{M}(f)(t)\|_{W_r}\leq C_{r,q,\beta}(\|f(t)\|^{\alpha+1}_{W_r}+\|f(t)\|^{\alpha+1}_{w_q})(1+t^2)^{-\frac{\gamma}{2}}.  
        \end{split}
    \end{align}
\end{lemma}
\begin{proof}
    To prove \eqref{lem d and M d}, we let 
    \begin{align}\label{cov}
        \begin{split}
            z=\sqrt{1+t^2}\left(v-\frac{tx}{1+t^2} \right).
        \end{split}
    \end{align}
    By direct calculation, we see
    \begin{align*}
        \begin{split}
            1+|x-vt|^2+|v|^2=1+\frac{|x|^2}{1+t^2}+|z|^2.
        \end{split}
    \end{align*}
    Using this, we have
   \begin{align*}
        \begin{split}
            f(t,x,v)&\leq \frac{ \|f(t)\|_{W_r}}{(1+|x-vt|^2+|v|^2)^{\frac{r}{2}}}\cr
            &\leq C_s \|f(t)\|_{W_r}\left(1+\frac{|x|^2}{1+t^2}+|z|^2 \right)^{-\frac{r}{2}}\cr
            &\leq C_r\|f(t)\|_{W_r}(1+|z|^2)^{-\frac{r}{2}}.
        \end{split}
    \end{align*}
    Therefore, we integrate over $v$ on both sides to obtain
    \begin{align*}
        \begin{split}
            \rho_f(t,x)\leq C_r\|f(t)\|_{W_r}(1+t^2)^{-\frac{3}{2}}\int_{\mathbb{R}^3}(1+|z|^2)^{-\frac{r}{2}}dz\leq C_r\|f(t)\|_{W_r}(1+t^2)^{-\frac{3}{2}},
        \end{split}
    \end{align*}
    via 
    \begin{align*}
        \begin{split}
            \int_{\mathbb{R}^3}(1+|z|^2)^{-\frac{r}{2}}dz\leq C_{r}<\infty,\text{ for }r>3.        
        \end{split}
    \end{align*}
    Next, we derive the estimate for the collision frequency. Since $0\leq 2\beta/(q-3)<\alpha$, Lemma~\ref{lem:Perthame ineq}$_{(2)}$ and \eqref{lem d and M d} lead to
    \begin{align*}
        \begin{split}
            \nu_{\alpha,\beta}(f)&=\rho^{\alpha-\frac{2\beta}{q-3}}_f(\rho_f T^{\frac{q-3}{2}}_f)^{\frac{2\beta}{q-3}}\cr
            &\leq C_{q}\rho^{\alpha-\frac{2\beta}{q-3}}_f\|f(t)\|^{\frac{2\beta}{q-3}}_{w_q}\cr
            &\leq C_{r,q}\|f(t)\|^{\alpha-\frac{2\beta}{q-3}}_{W_r}\|f(t)\|^{\frac{2\beta}{q-3}}_{w_q}(1+t^2)^{-\frac{3\alpha}{2}+\frac{3\beta}{q-3}}.
        \end{split}
    \end{align*}
    The proof for \eqref{lem d and M M1} directly follows from Corollary~\ref{cor:Mx lemma} and \eqref{lem d and M nu}:
    \begin{align*}
        \begin{split}
            |w_q(v)(\nu_{\alpha,\beta}(f)\mathcal{M}(f))(t,x,v)|&\leq |\nu_{\alpha,\beta}(f)||w_q\mathcal{M}(f)|\cr
            &\leq C_{r,q}\|f(t)\|^{\alpha-\frac{2\beta}{q-3}}_{W_r}\|f(t)\|^{1+\frac{2\beta}{q-3}}_{w_q}(1+t^2)^{-\frac{3\alpha}{2}+\frac{3\beta}{q-3}}.
        \end{split}
    \end{align*}
    To prove \eqref{lem d and M M1-1}, it is enough to show that
    \begin{align*}
        \begin{split}
            \|\mathcal{M}(f)(t)\|_{W_r}\leq C_r\|f(t)\|_{W_r},\text{ for }r>5.
        \end{split}
    \end{align*}
    Recalling \eqref{cov}, and we let $F(z)$ by
    \begin{align*}
        \begin{split}
            F(z)=f\left(t,x,\frac{z}{\sqrt{1+t^2}}+\frac{tx}{1+t^2}\right).
        \end{split}
    \end{align*}
    Then, we can see that
    \begin{align*}
        \begin{split}
            \rho_F=(1+t^2)^{\frac{3}{2}}\rho_f,\quad U_F=\sqrt{1+t^2}\left(U_f-\frac{tx}{1+t^2} \right),\quad T_F=(1+t^2)T_f.
        \end{split}
    \end{align*}
    This implies
    \begin{align*}
        \begin{split}
            \mathcal{M}(F)(z)=\mathcal{M}(f)(t,x,v).
        \end{split}
    \end{align*}
    Using this, we observe
    \begin{align}\label{large r est}
        \begin{split}
            \sup_{v}W_r(t,x,v)\mathcal{M}(f)(t,x,v)&=\sup_z\left(1+\frac{|x|^2}{1+t^2}+|z|^2 \right)^{\frac{r}{2}}\mathcal{M}(F)(z)\cr
            &\leq C_r\left\{\left(1+\frac{|x|^2}{1+t^2} \right)^{\frac{r}{2}}\|\mathcal{M}(F)\|_{L^{\infty}_{z}}+\|\mathcal{M}(F)\|_{w_r} \right\}\cr
            &\leq  C_{r,q}\left\{\left(1+\frac{|x|^2}{1+t^2} \right)^{\frac{r}{2}}\|F\|_{L^{\infty}_{z}}+\|F\|_{w_r} \right\}\cr
            &\leq C_r\sup_z\left(1+\frac{|x|^2}{1+t^2}+|z|^2 \right)^{\frac{r}{2}}F(z)\cr
            &=C_r\sup_v W_r(t,x,v)f(t,x,v).
        \end{split}
    \end{align}
    Here we used Corollary~\ref{cor:Mx lemma} and the unweighted Maxwellian estimate such that 
    \begin{align*}
        \begin{split}
            &\|\mathcal{M}(F)\|_{L^{\infty}_z}\leq C\|F\|_{L^{\infty}_z},\cr
            &\|\mathcal{M}(F)\|_{w_r}\leq C_r\|F\|_{w_r},\text{ for }r>5.
        \end{split}
    \end{align*}
    Therefore, taking $\sup_x$ on both sides of \eqref{large r est}, we get \eqref{lem d and M M1-1}. Lastly, we derive \eqref{lem d and M M2}. By the triangle inequality, we have
    \begin{align*}
        \begin{split}
            |W^{\frac{1}{r}}_r(t,x,v)-W^{\frac{1}{r}}_r(t,x,U_f)|&\leq |(1,x-tv,v)-(1,x-tU_f,U_f)|=\sqrt{1+t^2}|v-U_f|.
        \end{split}
    \end{align*}
    This implies
    \begin{align*}
        \begin{split}
            W_r(t,x,v)\leq C_r(W_r(t,x,U_f)+(1+t^2)^{\frac{r}{2}}|v-U_f|^r).
        \end{split}
    \end{align*}
    Multiplying $W_r(t,x,v)$ to $\mathcal{M}(f)$, we get
    \begin{align}\label{WM est}
        \begin{split}
            W_r(t,x,v)\mathcal{M}(f)(t,x,v)&\leq C_r\frac{\rho_f}{T^{\frac{3}{2}}_f}\left(W_r(t,x,U)e^{-\frac{|v-U_f|^2}{2T_f}}+(1+t^2)^{\frac{r}{2}}|v-U|^re^{-\frac{|v-U_f|^2}{2T_f}}\right)\cr
            &\leq C_r\frac{\rho_f}{T^{\frac{3}{2}}_f}\left(W_r(t,x,U_f)+(1+t^2)^\frac{r}{2}T^{\frac{r}{2}}_f\right)\cr
            &\leq C_r\left(\frac{\rho_f}{T^{\frac{3}{2}}_f}W_r(t,x,U_f)+(1+t^2)^{\frac{r}{2}}\rho T^{\frac{r-3}{2}}_f\right),
        \end{split}
    \end{align}
    where we used $x^re^{-x^2}\leq C_r$ such that
    \begin{align*}
        \begin{split}
            |v-U_f|^r e^{-\frac{|v-U_f|^2}{2T_f}}\leq (2T_f)^{\frac{r}{2}}\left( \frac{|v-U_f|^2}{2T_f}\right)^{\frac{r}{2}}e^{-\frac{|v-U_f|^2}{2T_f}}\leq C_rT^{\frac{r}{2}}_f.
        \end{split}
    \end{align*}
    Now we claim 
    \begin{align}\label{claim WM}
        \begin{split}
            \frac{\rho_f}{T^{\frac{3}{2}}_f}W_r(t,x,U_f)\leq C_r\left[\|f(t)\|_{W_r}+(1+t^2)^{\frac{r}{2}}\rho_f T^{\frac{r-3}{2}}_f\right].        
        \end{split}
    \end{align}
    Inserting \eqref{claim WM} into \eqref{WM est} and multiplying $\nu_{\alpha,\beta}(f)$, we have
    \begin{align}\label{W rho M est}
        \begin{split}
            W_r(t,x,v)\nu_{\alpha,\beta}(f)\mathcal{M}(f)&\leq C_r\left[\rho^\alpha_f T^{\beta}_f\|f(t)\|_{W_r}+(1+t^2)^{\frac{r}{2}}\rho^{\alpha+1}_f T^{\beta+\frac{r-3}{2}}_f \right].
        \end{split}
    \end{align}
    For the second term, we observe
    \begin{align}\label{rho T est}
        \begin{split}
            \rho^{\alpha+1}_f T^{\beta+\frac{r-3}{2}}_f=\rho^{\alpha+1-\frac{2\beta+r-3}{q-3}}_f(\rho_f T_f^{\frac{q-3}{2}})^{\frac{2\beta+r-3}{q-3}}\leq C_{r,q}\rho^{\alpha+1-\frac{2\beta+r-3}{q-3}}_f\|f(t)\|^{\frac{2\beta+r-3}{q-3}}_{w_q},
        \end{split}
    \end{align}
    via Lemma~\ref{lem:Perthame ineq}$_{(2)}$. Plugging \eqref{rho T est} into \eqref{W rho M est}, \eqref{lem d and M d} and \eqref{lem d and M nu} yield
    \begin{align*}
        \begin{split}
            W_r(t,x,v)\nu_{\alpha,\beta}(f)\mathcal{M}(f)&\leq C_{r,q}\|f(t)\|^{\alpha+1-\frac{2\beta}{q-3}}_{W_r}\|f(t)\|^{\frac{2\beta}{q-3}}_{w_q}(1+t^2)^{-\frac{3\alpha}{2}+\frac{3\beta}{q-3}}\cr
            &+C_{r,q}\|f(t)\|^{\alpha+1-\frac{2\beta+r-3}{q-3}}_{W_r}\|f(t)\|^{\frac{2\beta+r-3}{q-3}}_{w_q}(1+t^2)^{-\frac{\gamma}{2}}.
        \end{split}
    \end{align*}
    Here we used
    \begin{align*}
        \begin{split}
            \frac{r}{2}-\frac{3}{2}\left(\alpha+1-\frac{2\beta+r-3}{q-3} \right)=-\frac{\gamma}{2}.
        \end{split}
    \end{align*}
    Then, we apply the AM-GM inequality to get
    \begin{align*}
        \begin{split}
            |W_r(t,x,v)\nu_{\alpha,\beta}(f)\mathcal{M}(f)|\leq C_{r,q,\alpha,\beta}[\|f(t)\|^{\alpha+1}_{W_r}+\|f(t)\|^{\alpha+1}_{w_q}](1+t^2)^{-\frac{\gamma}{2}}.  
        \end{split}
    \end{align*}
    It remains to prove the claim \eqref{claim WM}. To this end, we have
    \begin{align*}
        \begin{split}
            \int_{|v-U_f|>\sqrt{6T_f}}fdv&\leq \frac{1}{6T_f}\int_{|v-U_f|>\sqrt{6T_f}}|v-U_f|^2fdv\cr
            &\leq \frac{1}{6T_f}\int_{\mathbb{R}^3}|v-U_f|^2fdv\cr
            &\leq \frac{3\rho_f T_f}{6T_f}=\frac{\rho_f}{2}.
        \end{split}
    \end{align*}
    This gives
    \begin{align*}
        \begin{split}
            \int_{|v-U_f|\leq \sqrt{6T_f}}fdv=\rho_f- \int_{|v-U_f|>\sqrt{6T_f}}fdv\geq \frac{\rho_f}{2}.
        \end{split}
    \end{align*}
    We first consider the case $W^{\frac{2}{r}}_r(t,x,U_f)\geq 24(1+t^2)T_f$. For $|v-U_f|\leq \sqrt{6T_f}$, this condition implies
    \begin{align*}
        \begin{split}
            \sqrt{1+t^2}|v-U_f|\leq \sqrt{6(1+t^2)T_f}\leq \frac{1}{2}W^{\frac{1}{r}}_r(t,x,U_f),
        \end{split}
    \end{align*}
    and thereby
    \begin{align*}
        \begin{split}
    W^{\frac{1}{r}}_r(t,x,v)&\geq W^{\frac{1}{r}}_r(t,x,U_f)-\sqrt{1+t^2}|v-U_f|\cr
            &\geq \frac{1}{2}W^{\frac{1}{r}}_r(t,x,U_f).
        \end{split}
    \end{align*}
    Then, we evaluate
    \begin{align}\label{WM claim pf 1}
        \begin{split}
            \frac{\rho_f}{2}\leq \int_{ |v-U_f|\leq \sqrt{6T_f} }f(t,x,v)dv\leq \frac{2^r\|f(t)\|_{W_r}}{W_r(t,x,U_f)}\int_{ |v-U_f|\leq \sqrt{6T_f} }dv\leq C_r\frac{\|f(t)\|_{W_r}}{W_r(t,x,U_f)}T^{\frac{3}{2}}_f.
        \end{split}
    \end{align}
    Now we take into account the another case $W^{\frac{2}{r}}_r(t,x,U_f)\leq 24(1+t^2)T_f$. By direct calculation, we get
    \begin{align}\label{WM claim pf 2}
        \begin{split}
            \frac{\rho_f}{T^{\frac{3}{2}}_f}W_r(t,x,U_f)\leq C_r (1+t^2)^{\frac{r}{2}}\rho_f T^{\frac{r-3}{2}}_f.
        \end{split}
    \end{align}
\end{proof}
Combining \eqref{WM claim pf 1} and \eqref{WM claim pf 2}, we obtain \eqref{claim WM}, and this completes the proof of this lemma.

\section{Construction of the solution space}\label{S3}
In this section, we construct the Picard-type iteration of sequences that approximate \eqref{BGK model}. For $n\geq 1$, we consider the following system:
\begin{align}\label{appr system}
    \begin{split}
        &\partial_t f_{n+1}+v\cdot\nabla_x f_{n+1}=\nu_{\alpha,\beta}(f_n)(\mathcal{M}(f_n)-f_{n+1}),\cr
        &f_{n+1}(0,x,v)=f_0(x,v),\cr
        &f_1=f_0(x-vt,v),
    \end{split}
\end{align}
Then, \eqref{appr system} can be rewritten as
\begin{align}\label{sequence mild form}
    \begin{split}
        f_{n+1}(t,x,v)&=e^{-\int^t_0 \nu_{\alpha,\beta}(f_n)(\tau,x-v(t-\tau))d\tau}f_0(x-vt,v)\cr
        &+\int^t_0 e^{-\int^t_s \nu_{\alpha,\beta}(f_n)(\tau,x-v(t-\tau))d\tau}(\nu_{\alpha,\beta}(f_n)\mathcal{M}(f_n))(s,x-v(t-s),v)ds.
    \end{split}
\end{align}
Now we design the solution space to capture that $f_{n}$ lies in this space for all $n\geq 1$.
\begin{definition}\label{def:Sol space}
   Let
    \begin{align*}
        \begin{split}
            &\frac{1}{3}<\alpha\leq 1,\quad 0\leq \beta\leq 1,\quad 3\alpha+2\beta\geq 2,\cr
            &q>\max\left(3+\frac{2(1+\beta)}{\alpha},3+\frac{6\beta}{3\alpha-1} \right),\cr
            &r\in \left(3,3\alpha+2-\frac{3(3\alpha+2\beta-1)}{q} \right)\cup(5,\infty).
        \end{split}
    \end{align*}
    We define our solution space $\Omega$ with the metric $d(f,g)=\underset{t}{\sup}\|f-g\|_{L^1_2(\mathbb{R}^6)}$ as
    \begin{align*}
        \begin{split}
            \Omega=\{f\in C([0,\infty);L^1_2(\mathbb{R}^6))\;|\; f\text{ satisfies }(A_1)-(A_4) \},
        \end{split}
    \end{align*}
    where conditions $(A_1)-(A_4)$ are given by
    \begin{itemize}
        \item $(A_1):$ $f$ is nonnegative.
        \item $(A_2):$ $f$ is uniformly bounded for sufficiently small $\varepsilon>0$:
        \begin{align*}
            \begin{split}
            \sup_t\|f(t)\|_{w_q}\leq \varepsilon,\quad\sup_t\|f(t)\|_{W_r}\leq \varepsilon.
            \end{split}
        \end{align*}
        \item $(A_3):$ The collision frequency $\nu_{\alpha,\beta}(f)$ has the dispersive effect:
        \begin{align*}
            \begin{split}
                \nu_{\alpha,\beta}(f)\leq C_{r,q}\varepsilon^\alpha(1+t^2)^{-\frac{3\alpha}{2}+\frac{3\beta}{q-3}}.
            \end{split}
        \end{align*}
        \item $(A_4):$ $f$ is uniform bounded in $L^1_2$ for sufficiently small $\varepsilon>0$:
    \begin{align*}
        \begin{split}
            \sup_t\|f(t)\|_{L^1_2}\leq 2\|f_0\|_{L^1_2}.
        \end{split}
    \end{align*}
    \end{itemize}
\end{definition}

In the remaining part, we conclude this section by showing that every $f_n$ belongs to the solution space $\Omega$.
\begin{proposition}\label{prop:Inv prop}
   Assume that $f^n$ belongs to $\Omega$. Then, for sufficiently small $\varepsilon$, $f^{n+1}$ lies in $\Omega$, that is, $f_{n+1}$ satisfies conditions $(A_1)-(A_4)$ in Definition \ref{def:Sol space}.
\end{proposition}
\begin{proof}
    Let $f^{n+1}$ be the solution to \eqref{appr system}, defined by \eqref{sequence mild form}.\\
    $\bullet$ \textbf{Condition $A_1:$ non-negativity.} The non-negativity follows directly from \eqref{sequence mild form}.\\
    $\bullet$ \textbf{Condition $A_2:$ Uniform boundedness in weighted $L^{\infty}$ norms.} Multiplying $w_q(v)$ on both sides of \eqref{sequence mild form}, we have
    \begin{align*}
        \begin{split}
            |w_q(v)f_{n+1}(t,x,v)|&\leq |w_q(v)f_0(x-vt,v)|+\int^t_0 |w_q(v)(\nu_{\alpha,\beta}(f_n)\mathcal{M}(f_n))(s,x-v(t-s),v)|ds\cr
            &\equiv I_1 + I_2,
        \end{split}
    \end{align*}
    where we used $e^{-\int^t_0\nu_{\alpha,\beta}(f_n)d\tau}\leq 1$ since $\nu_{\alpha,\beta}(f_n)\geq 0$ by condition $A_1$.
    For $I_1$, the assumption on the initial condition gives
    \begin{align*}
        \begin{split}
            I_1\leq \sup_{x,v\in\mathbb{R}^3}|w_q(v)f_0(x,v)|\leq \varepsilon^2.
        \end{split}
    \end{align*}
    For $I_2$, condition $(A2)$ and \eqref{lem d and M M1} in Lemma~\ref{lem:density and Maxwellian est} lead to
    \begin{align*}
        \begin{split}
            I_2&\leq \int^t_0 \|\nu_{\alpha,\beta}(f_n)\mathcal{M}(f_n)(s)\|_{w_q}ds\leq C_{r,q}\int^t_0\|f_n(s)\|^{\alpha-\frac{2\beta}{q-3}}_{W_r}\|f_n(s)\|^{1+\frac{2\beta}{q-3}}_{w_q}(1+s^2)^{-\frac{3\alpha}{2}+\frac{3\beta}{q-3}}ds\cr
            &\leq C_{r,q}\varepsilon^{\alpha+1}\int^{\infty}_0 (1+s^2)^{-\frac{3\alpha}{2}+\frac{3\beta}{q-3}}ds\cr
            &\leq C_{r,q,\alpha,\beta}\varepsilon^{\alpha+1},
        \end{split}
    \end{align*}
    where we used the integrability of $s$-integral since $3\alpha-\frac{6\beta}{q-3}>1$.  Combining the estimates for $I_1$ and $I_2$, we have
    \begin{align*}
        \begin{split}
            \|f_{n+1}(t)\|_{w,q}\leq \varepsilon^2+C_{r,q,\alpha,\beta}\varepsilon^{\alpha+1}\leq \varepsilon,
        \end{split}
    \end{align*}
    for sufficiently small $\varepsilon>0$. On the other hand, we multiply $W_r(t,x,v)$ on both sides of \eqref{sequence mild form} to get
    \begin{align*}
        \begin{split}
            |W_r(t,x,v)f_{n+1}(t,x,v)|&\leq |W_r(0,x-vt,v)f_0(x-vt,v)|\cr
            &+\int^t_0 |W_r(s,x-v(t-s),v)(\nu_{\alpha,\beta}(f_n)\mathcal{M}(f_n))(s,x-v(t-s),v)|ds\cr
            &\equiv J_1 + J_2,
        \end{split}
    \end{align*}
    where we used $e^{-\int^t_0\nu_{\alpha,\beta}(f_n)d\tau}\leq 1$ and the invariance property of $W_r(t,x,v)$ on the characteristic:
    \begin{align*}
        \begin{split}
            \frac{d}{ds}W_r(s,x-v(t-s),v)=0,\text{ for }s\in[0,t].
        \end{split}
    \end{align*}
    The estimate for $J_1$ comes from the hypothesis on the initial data as follows:
    \begin{align*}
        \begin{split}
            J_1\leq \sup_{x,v\in\mathbb{R}^3}|W_r(0,x,v)f_0(x,v)|\leq \varepsilon^2.
        \end{split}
    \end{align*}
    We derive $J_2$ using \eqref{lem d and M M1-1} and \eqref{lem d and M M2} in  Lemma~\ref{lem:density and Maxwellian est} and condition $(A_2)$. For $r>5$, we have
    \begin{align*}
        \begin{split}
            J_2&\leq \int^t_0 \|\nu_{\alpha,\beta}(f_n)\mathcal{M}(f_n)(s)\|_{W_r}ds\cr
            &\leq C_{r,q}\int^t_0\|f_n(s)\|^{\alpha+1-\frac{2\beta}{q-3}}_{W_r}\|f_n(s)\|^{\frac{2\beta}{q-3}}_{w_q}(1+s^2)^{-\frac{3\alpha}{2}+\frac{3\beta}{q-3}}ds\cr
            &\leq C_{r,q}\varepsilon^{\alpha+1}\int^{\infty}_0 (1+s^2)^{-\frac{3\alpha}{2}+\frac{3\beta}{q-3}}ds\cr
            &\leq C_{r,q,\alpha,\beta}\varepsilon^{\alpha+1}.
        \end{split}
    \end{align*}
    For $r\in(3,3\alpha+2-3(3\alpha+2\beta-1)/q)$, we get
    \begin{align*}
        \begin{split}
            J_2&\leq \int^t_0 \|\nu_{\alpha,\beta}(f_n)\mathcal{M}(f_n)(s)\|_{W_r}ds\cr
            &\leq C_{r,q}\int^t_0(\|f_n(s)\|^{\alpha+1}_{W_r}+\|f_n(s)\|^{\alpha+1}_{w_q})(1+s^2)^{-\frac{\gamma}{2}}ds\cr
            &\leq C_{r,q}\varepsilon^{\alpha+1}\int^{\infty}_0 (1+s^2)^{-\frac{\gamma}{2}}ds\cr
            &\leq C_{r,q,\gamma}\varepsilon^{\alpha+1}.
        \end{split}
    \end{align*}
    We collect the estimate for $J_1$ and $J_2$ to conclude
    \begin{align*}
        \begin{split}
            \|f_{n+1}(t)\|_{W_r}\leq \varepsilon^2+C_{r,q,\alpha,\beta,\gamma}\varepsilon^{\alpha+1}\leq \varepsilon,
        \end{split}
    \end{align*}
    for sufficiently small $\varepsilon>0$.
    \\
     $\bullet$ \textbf{Condition $A_3:$ Dispersive effect of the collision frequency.} We apply \eqref{lem d and M nu} in Lemma~\ref{lem:density and Maxwellian est} to deduce
     \begin{align*}
         \begin{split}
             \nu_{\alpha,\beta}(f_{n+1})(t,x)\leq C_{r,q} \|f_{n+1}(t)\|^{\alpha-\frac{2\beta}{q-3}}_{W_r}\|f_{n+1}(t)\|^{\frac{2\beta}{q-3}}_{w_q}(1+t^2)^{-\frac{3\alpha}{2}+\frac{3\beta}{q-3}}\leq C_r\varepsilon^\alpha(1+t^2)^{-\frac{3\alpha}{2}+\frac{3\beta}{q-3}}.
         \end{split}
     \end{align*}
     $\bullet$ \textbf{Condition $A_4:$ Uniform boundedness in $L^1_2$.} Multiplying $(1+|v|^2)$ and integrating over $x$ and $v$ on \eqref{sequence mild form}, we have
     \begin{align*}
         \begin{split}
             \int_{\mathbb{R}^6}(1+|v|^2)f_{n+1}(t,x,v)dxdv&\leq \int_{\mathbb{R}^6}(1+|v|^2)f_0(x,v)dxdv\cr
             &+\int^{t}_0\int_{\mathbb{R}^6}(1+|v|^2)(\nu_{\alpha,\beta}(f_n)\mathcal{M}(f_n))(s,x,v)dxdvds,
         \end{split}
     \end{align*}
     where we used $e^{-\int^t_s\nu_{\alpha,\beta}(f_n)d\tau}\leq 1$ and the measure-preserving property of $(x,v)\rightarrow (x-v(t-s),v)$. The first term is bounded by
     \begin{align*}
         \begin{split}
            \int_{\mathbb{R}^6}(1+|v|^2)f_0(x,v)dxdv\leq \|f_0\|_{L^1_2}.   
         \end{split}
     \end{align*}
     For the second term, we utilize conditions $(A_2)$, $(A_3)$, $(A_4)$, and the cancellation property of the Maxwellian such that
     \begin{align*}
         \begin{split}
          &\int^{t}_0\int_{\mathbb{R}^6}(1+|v|^2)(\nu_{\alpha,\beta}(f_n)\mathcal{M}(f_n))(s,x,v)dxdvds\cr
          &\leq C_{r,q} \int^t_0(1+s^2)^{-\frac{3\alpha}{2}+\frac{3\beta}{q-3}}\|f_{n}(s)\|^{\alpha-\frac{2\beta}{q-3}}_{W_r}\|f_{n}(s)\|^{\frac{2\beta}{q-3}}_{w_q}\left(\int_{\mathbb{R}^6}(1+|v|^2)f_n(s,x,v)dxdv\right)ds\cr
          &\leq C_{r,q,\alpha,\beta}\varepsilon^\alpha\sup_s\|f_n(s)\|_{L^1_2}.
         \end{split}
     \end{align*}
     Combining these, we conclude
     \begin{align*}
         \begin{split}
             \sup_t\|f_{n+1}(t)\|_{L^1_2}\leq (1+C_{r,q,\alpha,\beta}\varepsilon^\alpha)\|f_0\|_{L^1_2}\leq 2\|f_0\|_{L^1_2},
         \end{split}
     \end{align*}
     for sufficiently small $\varepsilon>0$. This completes the proof.
\end{proof}

\begin{corollary}
    Let $f_n$ be defined in \eqref{appr system}. Then, $f_n$ lies in $\Omega$ for all $n\geq 1$.
\end{corollary}
\begin{proof}
    Since $f_1=f_0(x-vt,v)$, we can easily show that $f_1\in \Omega$. Then, Proposition~\ref{prop:Inv prop} gives the desired result via inductive argument.
\end{proof}

\section{Lipschitz property for the relaxation operator}\label{S4}
In this section, we provide the Lipschitz property of the gain operator $\nu_{\alpha,\beta}(f)\mathcal{M}(f)$ and the loss operator $\nu_{\alpha,\beta}(f)f$. We then establish the limit function $f$ which satisfies $f_n\rightarrow f$ in $L^1_2$ using the Cauchy estimate. Throughout this section, we assume the parameter conditions
\begin{align*}
        \begin{split}
            &\frac{1}{3}<\alpha\leq 1,\quad 0\leq \beta\leq 1,\quad 3\alpha+2\beta\geq 2,\cr
            &q>\max\left(3+\frac{2(1+\beta)}{\alpha},3+\frac{6\beta}{3\alpha-1} \right),\cr
            &r\in \left(3,3\alpha+2-\frac{3(3\alpha+2\beta-1)}{q} \right)\cup(5,\infty).
        \end{split}
    \end{align*}
\begin{lemma}\label{element property of gaussian lemma}
    Let $\mathcal{M}(f)$ is defined in \eqref{Maxwellian def}. Then, we have
    \begin{align*}
        \begin{split}
            &(1)\;\int_{\mathbb{R}^3} \mathcal{M}(f)(1+|v|^2)dv\leq C\rho_f(1+|U_f|^2+T_f).\cr
            &(2)\;T^{\frac{1}{2}}_f\int_{\mathbb{R}^3} \left|\frac{\partial \mathcal{M}(f)}{\partial U_f}\right|(1+|v|^2)dv\leq C\rho_f(1+|U_f|^2+T_f).\cr
            &(3)\;T_f\int_{\mathbb{R}^3} \left|\frac{\partial\mathcal{M}(f)}{\partial T_f}\right|(1+|v|^2)dv\leq C\rho_f(1+|U_f|^2+T_f).
        \end{split}
    \end{align*}
\end{lemma}
\begin{proof}
    The first assertion directly follows from the mass and energy conservation:
    \begin{align*}
        \begin{split}
            \int_{\mathbb{R}^3}\mathcal{M}(f)(1+|v|^2)dv=\int_{\mathbb{R}^3}f(1+|v|^2)dv=\rho_f+(\rho|U_f|^2+3\rho_f T_f).
        \end{split}
    \end{align*}
    To show the second and third assertions, we let $v=U_f+\sqrt{T_f}w$ for $w\in\mathbb{R}^3$. Then, the Maxwellian $\mathcal{M}(w)$ can be written as
    \begin{align*}
        \begin{split}
            \mathcal{M}(w)=\frac{\rho}{(2\pi T)^{\frac{3}{2}}}\exp\left(-\frac{|w|^2}{2}\right).
        \end{split}
    \end{align*}
    Using this, we have
    \begin{align*}
        \begin{split}
            T^{\frac{1}{2}}_f\int_{\mathbb{R}^3}\left|\frac{\partial \mathcal{M}(f)}{\partial U_f} \right|(1+|v|^2)dv&\leq T^{\frac{1}{2}}_f\int_{\mathbb{R}^3}\mathcal{M}(f)\frac{|v-U_f|}{T_f}(1+|v|^2)dv\cr
            &\leq \frac{\rho_f}{(2\pi)^{\frac{3}{2}}}\int_{\mathbb{R}^3}|w|(1+|U_f+\sqrt{T_f}w|^2)\exp\left(-\frac{|w|^2}{2}\right)dw\cr
            &\leq C\rho_f(1+|U_f|^2 +T_f)\int_{\mathbb{R}^3}(|w|+|w|^3)e^{-\frac{|w|^2}{2}}dw\cr
            &\leq C\rho_f(1+|U_f|^2 +T_f),
        \end{split}
    \end{align*}
    where we used
    \begin{align*}
        \begin{split}
            \left|\frac{\partial \mathcal{M}(f)}{\partial U_f}\right|=\frac{|v-U_f|}{T_f}\mathcal{M}(f).
        \end{split}
    \end{align*}
    Lastly, we deduce
    \begin{align*}
        \begin{split}
            T_f\int_{\mathbb{R}^3} \left|\frac{\partial\mathcal{M}(f)}{\partial T_f}\right|(1+|v|^2)dv&\leq T_f\int_{\mathbb{R}^3}\left|\frac{|v-U_f|^2}{2T^2_f}-\frac{3}{2T_f} \right|\mathcal{M}(f)(1+|v|^2)dv\cr
            &=\frac{\rho_f T_f}{(2\pi)^{\frac{3}{2}}}\int_{\mathbb{R}^3}\left| \frac{|w|^2}{2T_f}-\frac{3}{2T_f}\right|(1+|U_f+\sqrt{T_f}w|^2)e^{-\frac{|w|^2}{2}}dw\cr
            &\leq C\rho_f(1+|U_f|^2+T_f)\int_{\mathbb{R}^3}(1+|w|^4)e^{-\frac{|w|^2}{2}}dw\cr
            &\leq C\rho_f(1+|U_f|^2+T_f).
        \end{split}
    \end{align*}
    We utilized the following direct calculation:
    \begin{align*}
        \begin{split}
            \left|\frac{\partial \mathcal{M}(f)}{\partial T_f} \right|=\left|\frac{|v-U_f|^2}{2T^2_f}-\frac{3}{2T_f} \right|\mathcal{M}(f).
        \end{split}
    \end{align*}
    This completes the proof.
\end{proof}

\begin{proposition}\label{prop:Lip prop}
    Let $f$ and $g$ belong to $\Omega$. Then, we have
    \begin{align*}
        \begin{split}
            \|\nu_{\alpha,\beta}(f)\mathcal{M}(f)-\nu_{\alpha,\beta}(g)\mathcal{M}(g)\|_{L^1_2}+\|\nu_{\alpha,\beta}(f)f-\nu_{\alpha,\beta}(g)g\|_{L^1_2}\leq C_{Lip}\varepsilon^\alpha\|f-g\|_{L^1_2(\mathbb{R}^6)},
        \end{split}
    \end{align*}
    where $C_{Lip}$ is some constant which depends on $q$, $\alpha$, and $\beta$.
\end{proposition}
\begin{proof}
    We define $h$ as a convex combination of $f$ and $g$ such that
    \begin{align*}
        \begin{split}
            h=\theta f+(1-\theta)g,\quad \theta\in(0,1).
        \end{split}
    \end{align*}
    Further, we define the moments of $h$, and its associated collision frequency $\nu_h$ and the Maxwellian $M_h$ as 
    \begin{align*}
        \begin{split}
            &(\rho_h,p_h,E_h)=(\rho_h,\rho_hU_h,\rho_h(|U_h|^2+3T_h))=\int_{\mathbb{R}^3}(1,v,|v|^2)hdv,\cr
            &\nu_h=\rho^\alpha_hT^{\beta}_h,\cr
            &\mathcal{M}_h=\frac{\rho_h}{(2\pi T_h)^{\frac{3}{2}}}\exp\left( -\frac{|v-U_h|^2}{2T_h}\right).
        \end{split}
    \end{align*}
    $\bullet$ \textbf{Estimate for the gain term:} Taylor's theorem gives
    \begin{align}\label{Taylor expansion}
        \begin{split}
            |\nu_{\alpha,\beta}(f)\mathcal{M}(f)-\nu_{\alpha,\beta}(g) \mathcal{M}(g)|&\leq \int^1_0\left| \frac{d}{d\theta}(\nu_h\mathcal{M}_h)\right|d\theta\cr
            &\leq |\rho_f-\rho_g|\int^1_0\left|\frac{\partial (\nu_h\mathcal{M}_{h})}{\partial\rho_h}\right|d\theta\cr
            &+|p_f-p_g|\int^1_0\left|\frac{\partial (\nu_h\mathcal{M}_{h})}{\partial p_h}\right|d\theta\cr
            &+|E_f-E_g|\int^1_0\left|\frac{\partial (\nu_h\mathcal{M}_{h})}{\partial E_h}\right|d\theta.
        \end{split}
    \end{align}
    We first observe that
    \begin{align}\label{Lip est 1}
        \begin{split}
            &\int_{\mathbb{R}^3}|\rho_f-\rho_g|dx\leq \int_{\mathbb{R}^6} |f-g|dvdx\leq \|f-g\|_{L^1_2},\cr
            &\int_{\mathbb{R}^3}|p_f-p_g|dx\leq \int_{\mathbb{R}^6} |v||f-g|dvdx\leq \|f-g\|_{L^1_2}\cr
            &\int_{\mathbb{R}^3}|E_f-E_g|dx\leq \int_{\mathbb{R}^6} |v|^2|f-g|dvdx\leq \|f-g\|_{L^1_2}.
        \end{split}
    \end{align}
    Moreover, we have
    \begin{align*}
        \begin{split}
            \|h\|_{w_q}\leq \theta\|f\|_{w_q} + (1-\theta)\|g\|_{w_q}\leq \varepsilon.
        \end{split}
    \end{align*}
    By direct calculation, we have
    \begin{align*}
        \begin{split}
            &\frac{\partial(\nu_h \mathcal{M}_{h})}{\partial \rho_h}\cr
            &=\rho^{\alpha-1}_h\left[\left((\alpha+1-\beta)T^{\beta}_h+\frac{\beta}{3}|U_h|^2T^{\beta-1}_h \right)\mathcal{M}_h-T^{\beta}_hU_h\cdot \frac{\partial \mathcal{M}_h}{\partial U_h}+T^{\beta}_h\left(\frac{|U_h|^2}{3}-T_h \right)\frac{\partial\mathcal{M}_h}{\partial T_h}\right],\cr
            &\frac{\partial(\nu_h \mathcal{M}_{h})}{\partial p_h}=\rho^{\alpha-1}_h\left[T^{\beta}_h\frac{\partial \mathcal{M}_h}{\partial U_h}-\frac{2}{3}T^{\beta}_h U_h \frac{\partial\mathcal{M}_h}{\partial T_h}-\frac{2\beta}{3}T^{\beta-1}_h U_h \mathcal{M}_h\right],\cr
            &\frac{\partial(\nu_h \mathcal{M}_{h})}{\partial E_h}=\rho^{\alpha-1}_h\left[\frac{1}{3}T^{\beta}_h\frac{\partial\mathcal{M}_h}{\partial T_h}+\frac{\beta}{3}T^{\beta-1}_h\mathcal{M}_h\right].
        \end{split}
    \end{align*}
    Then, we apply Lemma \ref{element property of gaussian lemma} to obtain
    \begin{align*}
        \begin{split}
    &\int_{\mathbb{R}^3}(1+|v|^2)\left(\left| \frac{\partial(\nu_h \mathcal{M}_{h})}{\partial \rho_h}\right|+\left| \frac{\partial(\nu_h \mathcal{M}_{h})}{\partial p_h}\right|+ \left| \frac{\partial(\nu_h \mathcal{M}_{h})}{\partial E_h}\right|\right)dv\cr
            &\leq C_{\alpha,\beta}\rho^\alpha_hT^{\beta}_h(1+|U_h|^2+T_h)\left(1+\frac{1+|U_h|^2}{T_h}\right)\cr
            &\leq C_{\alpha,\beta}\{\rho^\alpha_hT^{\beta}_h(1+|U_h|^2+T_h)+ \rho^\alpha_hT^{\beta-1}_h+\rho^\alpha_h|U_h|^3T^{\beta-1}_h+\rho^\alpha_h|U_h|^4T^{\beta-1}_h\}.
        \end{split}
    \end{align*}
    Therefore, Lemma \ref{lem:Perthame ineq}$_{(4)-(6)}$ leads to
    \begin{align}\label{Lip est 2}
        \begin{split}
            &\int_{\mathbb{R}^3}(1+|v|^2)\left(\left| \frac{\partial(\nu_h \mathcal{M}_{h})}{\partial \rho_h}\right|+\left| \frac{\partial(\nu_h \mathcal{M}_{h})}{\partial p_h}\right|+ \left| \frac{\partial(\nu_h \mathcal{M}_{h})}{\partial E_h}\right|\right)dv\leq C_{q,\alpha,\beta}\|h\|^\alpha_{w_q}\leq C_{q,\alpha,\beta}\varepsilon^\alpha.
        \end{split}
    \end{align}
    Multiplying \eqref{Taylor expansion} by $(1+|v|^2)$ and integrating over $x$ and $v$, we utilize \eqref{Lip est 1} and \eqref{Lip est 2} to obtain
    \begin{align*}
        \begin{split}
            \int_{\mathbb{R}^6}(1+|v|^2)|\nu_{\alpha,\beta}(f)\mathcal{M}(f)-\nu_{\alpha,\beta}(g)\mathcal{M}(g)|dvdx\leq C_{q,\alpha,\beta}\varepsilon^\alpha\|f-g\|_{L^1_2}.
        \end{split}
    \end{align*}
    $\bullet$ \textbf{Estimate for the loss term:} We first write
    \begin{align}\label{Taylor expansion 2}
        \begin{split}
            \int_{\mathbb{R}^6}(1+|v|^2)|\nu_{\alpha,\beta}(f)f-\nu_{\alpha,\beta}(g)g|dxdv&\leq \int_{\mathbb{R}^6}(1+|v|^2)\int^1_0\left|\frac{d}{d\theta}(\rho^\alpha_hT^{\beta}_h h)\right|d\theta dxdv\cr
            &\leq \int_{\mathbb{R}^6}(1+|v|^2)|\rho^\alpha_hT^{\beta}_h||f-g|dxdv\cr
            &+\int_{\mathbb{R}^3}\left(\int_{\mathbb{R}^3}(1+|v|^2)hdv\right)\int^1_0\left|\frac{d}{d\theta}(\rho^\alpha_hT^{\beta}_h)d\theta\right| dx\cr
            &\equiv I_1 + I_2.
        \end{split}
    \end{align}
    For $I_1$, Lemma~\ref{lem:Perthame ineq}$_{(6)}$ gives
    \begin{align*}
        \begin{split}
            I_1\leq C_{q,\alpha}\int^1_0\|h\|^\alpha_{w_q}d\theta\|f-g\|_{L^1_2}\leq C_{q,\alpha}\varepsilon^\alpha\|f-g\|_{L^1_2}.  
        \end{split}
    \end{align*}
    On the other hand, we observe the following from direct calculation:
    \begin{align*}
        \begin{split}
            \frac{d}{d\theta}(\rho^\alpha_hT^{\beta}_h)&=\rho^{\alpha-1}_h\left((\alpha-\beta)T^{\beta}_h+\frac{\beta}{3}|U_h|^2T^{\beta-1}_h \right)(\rho_f-\rho_g)\cr
            &-\frac{2\beta}{3}\rho^{\alpha-1}_hT^{\beta-1}_hU_h\cdot (p_f-p_g)\cr
            &+\frac{\beta}{3}\rho^{\alpha-1}_hT^{\beta-1}_{h}(E_f-E_g).
        \end{split}
    \end{align*}
    Then, we derive $I_2$ as follows:
    \begin{align*}
        \begin{split}
            I_2&\leq C_{\alpha,\beta}\int^1_0\int_{\mathbb{R}^3} \rho^\alpha_h(1+|U_h|^2+T_h)\{T^{\beta}_h+T^{\beta-1}_h(1+|U_h|^2)\}(|\rho_f-\rho_g|+|p_f-p_g|+|E_f-E_g|)dxd\theta\cr
            &\leq C_{q,\alpha,\beta}\int^1_0\|h\|^\alpha_{w_q}d\theta\int_{\mathbb{R}^3}(|\rho_f-\rho_g|+|p_f-p_g|+|E_f-E_g|)dx\cr
            &\leq C_{q,\alpha,\beta}\varepsilon^\alpha\|f-g\|_{L^1_2},
        \end{split}
    \end{align*}
    where we used 
    \begin{align*}
        \begin{split}
            \int_{\mathbb{R}^3} (1+|v|^2)hdv=\rho_h(1+|U_h|^2+3T_h),
        \end{split}
    \end{align*}
    and Lemma~\ref{lem:Perthame ineq}$_{(4)-(6)}$ such that
    \begin{align*}
        \begin{split}
            &\rho^\alpha_h(1+|U_h|^2+T_h)\{T^{\beta}_h+T^{\beta-1}_h(1+|U_h|^2)\}\cr
            &=\rho^\alpha_h T^{\beta+1}_h +2\rho^\alpha_h T^{\beta}(1+|U_h|^2)+\rho^\alpha_h T^{\beta-1}_h (1+|U_h|^2)^2\cr
            &\leq C_{q}\|h\|^\alpha_{w_q}.
        \end{split}
    \end{align*}
    Inserting the estimates for $I_1$ and $I_2$ into \eqref{Taylor expansion 2}, we get the desired result.
\end{proof}

\begin{proposition}\label{prop:cauchy prop}
    Assume that the hypothesis of Theorem \ref{thm:main thm} holds, and let $\{f_n\}$ be the sequence defined in \eqref{appr system}. Then, there exists $f\in\Omega$ such that, for every $T^f>0$,
    \begin{align*}
        \begin{split}
            \sup_{0\leq t\leq T^f}\|f_n-f\|_{L^{1}_2(\mathbb{R}^6)}\rightarrow 0,\text{ as }n\rightarrow \infty.
        \end{split}
    \end{align*}
\end{proposition}
\begin{proof}
    It is enough to show that $\{f_n\}_{n\in\mathbb{N}}$ is a Cauchy sequence in $\Omega$. Let $F_{n+1}=f_{n+1}-f_n$. We subtract \eqref{appr system} for the $n$-th step from the $(n+1)$-th step to get 
    \begin{align*}
        \begin{split}
            &\partial_t F_{n+1}+v\cdot \nabla_xF_{n+1}+\nu_{\alpha,\beta}(f_n)F_{n+1}\cr
            &=(\nu_{\alpha,\beta}(f_n)\mathcal{M}(f_n)-\nu_{\alpha,\beta}(f_{n-1})\mathcal{M}(f_{n-1}))-(\nu_{\alpha,\beta}(f_n)-\nu_{\alpha,\beta}(f_{n-1}))f_n.   
        \end{split}
    \end{align*}
    Integrating the above equation along the characteristic, we have
    \begin{align*}
        \begin{split}
            F_{n+1}(t,x,v)&=\int^t_0 e^{-\int^t_s \nu_{\alpha,\beta}(f_n)d\tau}(\nu_{\alpha,\beta}(f_n)\mathcal{M}(f_n)-\nu_{\alpha,\beta}(f_{n-1})\mathcal{M}(f_{n-1}))(s,x-v(t-s),v)ds\cr
            &-\int^t_0 e^{-\int^t_s \nu_{\alpha,\beta}(f_n)d\tau}(\nu_{\alpha,\beta}(f_n)-\nu_{\alpha,\beta}(f_{n-1}))f_n(s,x-v(t-s),v)ds.
        \end{split}
    \end{align*}
    Now we multiply $(1+|v|^2)$ on both sides and integrate over $(x,v)\in\mathbb{R}^6$. Then, we get
    \begin{align}\label{F cauchy est 1}
        \begin{split}
             \int_{\mathbb{R}^6}(1+|v|^2)|F_{n+1}(t,x,v)|dxdv&\leq \int^t_0 \|\nu_{\alpha,\beta}(f_n)\mathcal{M}(f_n)-\nu_{\alpha,\beta}(f_{n-1})\mathcal{M}(f_{n-1}))(s)\|_{L^1_2}ds\cr
            &+\int^t_0 \int_{\mathbb{R}^6}(1+|v|^2)|\nu_{\alpha,\beta}(f_n)-\nu_{\alpha,\beta}(f_{n-1})|f_n(s,x,v)dxdvds\cr
            &\equiv I_1+I_2,
        \end{split}
    \end{align}
    where we used the measure-preserving property on $(x-v(t-s),v)\rightarrow (x,v)$. For $I_1$, Proposition~\ref{prop:Lip prop} leads to
    \begin{align*}
        \begin{split}
            I_1\leq C_{Lip}\varepsilon^\alpha\int^t_0 \|F_n(s)\|_{L^1_2}ds. 
        \end{split}
    \end{align*}
    For $I_2$, we have
    \begin{align*}
        \begin{split}
            I_2&\leq \int^t_0\int_{\mathbb{R}^6}(1+|v|^2)|\nu_{\alpha,\beta}(f_n)f_n-\nu_{\alpha,\beta}(f_{n-1})f_{n-1}|dxdvds\cr
            &+\int^t_0\int_{\mathbb{R}^6}(1+|v|^2)|\nu_{\alpha,\beta}(f_{n-1})||f_n-f_{n-1}|dxdvds\cr
            &\leq C_{Lip}\varepsilon^\alpha\int^t_0\|F_n(s)\|_{L^1_2}ds+C_{r,q}\varepsilon^\alpha\int^t_0\|F_n(s)\|_{L^1_2}ds,
        \end{split}
    \end{align*}
    where we used Proposition~\ref{prop:Lip prop} and $\nu_{\alpha,\beta}(f)\leq C_{r,q,\alpha}\varepsilon^\alpha$ if $f\in\Omega$. Plugging the estimate for $I_1$ and $I_2$ into \eqref{F cauchy est 1}, we deduce
    \begin{align}\label{F Cauchy est 2}
        \begin{split}
            \|F_{n+1}(t)\|_{L^1_2}\leq (2C_{Lip}+C_{r,q,\alpha})\varepsilon^\alpha\int^t_0 \|F_n(s)\|_{L^{1}_2}ds.
        \end{split}
    \end{align}
    Let $(2C_{Lip}+C_{r,q,\alpha})\varepsilon^\alpha T^f=C_*$. For $n\geq 1$, we iterate \eqref{F Cauchy est 2} to obtain
    \begin{align*}
        \begin{split}
            \sup_{0\leq t\leq T^f}\|F_{n+1}(t)\|_{L^1_2}\leq \frac{C^{n-1}_*}{(n-1)!}\sup_{0\leq s\leq T^f}\|F_2(s)\|_{L^1_2}  \leq \frac{C^{n-1}_*}{(n-1)!}(\|f_2\|_{L^1_2}+\|f_1\|_{L^1_2})\leq \frac{4\|f_0\|_{L^1_2}C^{n-1}_*}{(n-1)!}.
        \end{split}
    \end{align*}
    Therefore, for $m>n\geq 2$, the following estimate holds:
\begin{align*}
    \begin{split}
        \sup_{0\leq t\le T^f}\|(f_m-f_n)(t)\|_{L^1_2(\mathbb{R}^6)}\leq 4\|f_0\|_{L^1_2}\left(e^{C_*}-\sum^{n-2}_{k=0}\frac{C^k_*}{k!} \right)\rightarrow 0,\text{ as }m,n\rightarrow \infty.
    \end{split}
\end{align*}
Thus, $\{f_n\}$ is Cauchy in $C([0,T^f];L^1_2)$. Since the iterates are defined on the entire half-line in $t$, their limits on different finite time intervals agree. That is, they determine a single function $f\in C([0,\infty_;L^1_2])$ such that
\begin{align*}
    \begin{split}
        \sup_{0\leq t\leq T^f}\|(f_n-f)(t)\|_{L^1_2}\rightarrow 0,\text{ for every }T^f>0.
    \end{split}
\end{align*} 
\end{proof}

\section{Proof for Theorem~\ref{thm:main thm}}\label{S5}
In the last section, we prove Theorem~\ref{thm:main thm}. By Proposition~\ref{prop:cauchy prop}, we have already established that there exists $f\in\Omega$ which satisfies $f^n\rightarrow f$ in $L^1_2$. Therefore, it is enough to show that the nonlinear source in \eqref{appr system}; $\nu_{\alpha,\beta}(f_n)(\mathcal{M}(f_n)-f_{n+1})$ converges to $\nu_{\alpha,\beta}(f)(\mathcal{M}(f)- f)$. From Proposition~\ref{prop:Lip prop}, we have
\begin{align*}
    \begin{split}
        \|\nu_{\alpha,\beta}(f_n)\mathcal{M}(f_n)-\nu_{\alpha,\beta}(f)\mathcal{M}(f)\|_{L^1_2}\leq C_{Lip}\varepsilon^\alpha\|f_n- f\|_{L^1_2}\rightarrow 0,\text{ uniformly in }t.
    \end{split}
\end{align*}
Moreover, we have
\begin{align*}
    \begin{split}
        \|\nu_{\alpha,\beta}(f_n)f_{n+1}-\nu_{\alpha,\beta}(f) f\|_{L^1_2}&\leq \|\nu_{\alpha,\beta}(f_{n+1})f_{n+1}-\nu_{\alpha,\beta}(f_n)f_n\|_{L^1_2}\cr
        &+\|\nu_{\alpha,\beta}(f_n)(f_{n+1}-f_n)\|_{L^1_2}\cr
        &+\|\nu_{\alpha,\beta}(f_{n+1})f_{n+1}-\nu_{\alpha,\beta}(f)f\|_{L^1_2}.
    \end{split}
\end{align*}
The first and third terms converge to zero by Proposition~\ref{prop:Lip prop}, and the second term is controlled by $f_n\in\mathcal{S}$ and \eqref{lem d and M nu} in Lemma~\ref{lem:density and Maxwellian est}
\begin{align*}
    \begin{split}
        \|\nu_{\alpha,\beta}(f_n)(f_{n+1}-f_n)\|_{L^1_2}\leq |\nu_{\alpha,\beta}(f_n)|\|f_{n+1}-f_n\|_{L^1_2}\leq C_{r,q}\varepsilon^\alpha\|f_{n+1}-f_n\|_{L^1_2}\rightarrow 0.
    \end{split}
\end{align*}
Therefore, if we pass to the limit in the following formula, which originates from \eqref{appr system}
\begin{align*}
    \begin{split}
        f_{n+1}(t,x,v)=f_0(x-vt,v)+\int^t_0 \nu_{\alpha,\beta}(f_n)(\mathcal{M}(f_n)-f_{n+1})(s,x-v(t-s),v)ds.
    \end{split}
\end{align*}
Then, we can conclude that $f$ is the unique mild solution to \eqref{BGK model}, which satisfies Definition~\ref{def:Sol def}. Indeed, we can show the uniqueness of the solution. Let $f,g\in\Omega$ be two mild solutions with the same initial data $f_0$. Then, the difference between them satisfies
\begin{align*}
    \begin{split}
        (\partial_t +v\cdot \nabla_x+\nu_{\alpha,\beta}(f))(f-g)=(\nu_{\alpha,\beta}(f)\mathcal{M}(f)-\nu_{\alpha,\beta}(g)\mathcal{M}(g))-(\nu_{\alpha,\beta}(f)-\nu_{\alpha,\beta}(g))g.
    \end{split}
\end{align*}
Using the analogous argument in the proof of Proposition~\ref{prop:cauchy prop}, we deduce
\begin{align*}
    \begin{split}
        \|(f-g)(t)\|_{L^1_2}\leq C_q\varepsilon^\alpha\int^t_0 \|(f-g)(s)\|_{L^1_2}ds.
    \end{split}
\end{align*}
Then, Gronwall's inequality yields $f=g$. This completes the proof of Theorem~\ref{thm:main thm}.\\
\newline
\noindent{\bf Acknowledgment}\newline
\noindent This work was supported by the National Research Foundation of
Korea(NRF) grant funded by the Korea government(MSIT). (No.RS-2023-NR076676).\newline 
\newline
\noindent\textbf{Declaration of generative AI and AI-assisted technologies in the manuscript preparation process}\newline
During the preparation of this work, the author(s) used ChatGPT (OpenAI) for rephrasing and improving the readability of the Introduction section. The author(s) reviewed and edited the output as needed and take full responsibility for the content of the published article.
\bibliographystyle{abbrv}
\bibliography{ref} 

@article{bhatnagar1954model,
  title={A model for collision processes in gases. I. Small amplitude processes in charged and neutral one-component systems},
  author={Bhatnagar, Prabhu Lal and Gross, Eugene P and Krook, Max},
  journal={Physical review},
  volume={94},
  number={3},
  pages={511},
  year={1954},
  publisher={APS}
}

@article{chen2016global,
  title={Global existence and uniqueness to the {C}auchy problem of the {BGK} equation with infinite energy},
  author={Chen, Zili and Zhang, Xianwen},
  journal={Mathematical Methods in the Applied Sciences},
  volume={39},
  number={11},
  pages={3116--3135},
  year={2016},
  publisher={Wiley Online Library}
}

@article{chen2018smooth,
  title={Smooth solutions to the {BGK} equation and the {ES-BGK} equation with infinite energy},
  author={Chen, Zili},
  journal={Journal of Differential Equations},
  volume={265},
  number={1},
  pages={389--416},
  year={2018},
  publisher={Elsevier}
}

@article{guo2001vlasov,
  title={The {V}lasov--{P}oisson--{B}oltzmann system near vacuum},
  author={Guo, Yan},
  journal={Communications in Mathematical Physics},
  volume={218},
  pages={293--313},
  year={2001},
  publisher={Springer}
}

@article{illner1984boltzmann,
  title={The {B}oltzmann equation: global existence for a rare gas in an infinite vacuum},
  author={Illner, Reinhard and Shinbrot, Marvin},
  journal={Communications in mathematical physics},
  volume={95},
  number={2},
  pages={217--226},
  year={1984},
  publisher={Springer}
}

@article{bellomo1985cauchy,
  title={On the {C}auchy problem for the nonlinear {B}oltzmann equation global existence uniqueness and asymptotic stability},
  author={Bellomo, Nicola and Toscani, Giuseppe},
  journal={Journal of mathematical physics},
  volume={26},
  number={2},
  pages={334--338},
  year={1985},
  publisher={American Institute of Physics}
}

@article{toscani1986non,
  title={On the non-linear {B}oltzmann equation in unbounded domains},
  author={Toscani, G},
  journal={Archive for Rational Mechanics and Analysis},
  volume={95},
  pages={37--49},
  year={1986},
  publisher={Springer}
}

@article{glassey2006global,
  title={Global solutions to the {C}auchy Problem for the Relativistic {B}oltzmann Equation with Near--Vacuum Data.},
  author={Glassey, Robert T},
  journal={Communications in mathematical physics},
  volume={264},
  number={3},
  year={2006}
}

@article{ha20041,
  title={${L^1}$ stability of the {B}oltzmann equation for the hard-sphere model},
  author={Ha, Seung-Yeal},
  journal={Archive for rational mechanics and analysis},
  volume={173},
  pages={279--296},
  year={2004},
  publisher={Springer}
}

@article{ha2005nonlinear,
  title={Nonlinear functionals of the {B}oltzmann equation and uniform stability estimates},
  author={Ha, Seung-Yeal},
  journal={Journal of Differential Equations},
  volume={215},
  number={1},
  pages={178--205},
  year={2005},
  publisher={Elsevier}
}

@article{duan2006boltzmann,
  title={Boltzmann equation with external force and {V}lasov-{P}oisson-{B}oltzmann system in infinite vacuum},
  author={Duan, RJ and Yang, Tong and Zhu, CJ},
  journal={Discrete and Continuous Dynamical Systems},
  volume={16},
  number={1},
  pages={253},
  year={2006},
  publisher={Citeseer}
}

@article{duan2006l1,
  title={${L^1}$ stability for the {V}lasov--{P}oisson--{B}oltzmann system around vacuum},
  author={Duan, Renjun and Zhang, Mei and Zhu, Changjiang},
  journal={Mathematical Models and Methods in Applied Sciences},
  volume={16},
  number={09},
  pages={1505--1526},
  year={2006},
  publisher={World Scientific}
}

@article{chae2006time,
  title={Time-asymptotic behavior of the {V}lasov--{P}oisson--{B}oltzmann system near vacuum},
  author={Chae, Myeongju and Ha, Seung-Yeal and Hwang, Hyung Ju},
  journal={Journal of Differential Equations},
  volume={230},
  number={1},
  pages={71--85},
  year={2006},
  publisher={Elsevier}
}

@article{toscani1987h,
  title={H-theorem and asymptotic trend of the solution for a rarefied gas in the vacuum},
  author={Toscani, G},
  journal={Archive for Rational Mechanics and Analysis},
  volume={100},
  pages={1--12},
  year={1987},
  publisher={Springer}
}

@article{toscani1988global,
  title={Global solution of the initial value problem for the {B}oltzmann equation near a local {M}axwellian},
  author={Toscani, G},
  journal={Archive for Rational Mechanics and Analysis},
  volume={102},
  pages={231--241},
  year={1988},
  publisher={Springer}
}

@article{palczewski1989global,
  title={Global solution of the {B}oltzmann equation for rigid spheres and initial data close to a local {M}axwellian},
  author={Palczewski, A and Toscani, G},
  journal={Journal of mathematical physics},
  volume={30},
  number={10},
  pages={2445--2450},
  year={1989},
  publisher={American Institute of Physics}
}

@article{son2023bgk,
  title={The {ES-BGK} for the polyatomic molecules with infinite energy},
  author={Son, Sung-jun and Yun, Seok-Bae},
  journal={Journal of Statistical Physics},
  volume={190},
  number={8},
  pages={129},
  year={2023},
  publisher={Springer}
}

@article{mieussens2000discrete2,
  title={Discrete-velocity models and numerical schemes for the {B}oltzmann-{BGK} equation in plane and axisymmetric geometries},
  author={Mieussens, Luc},
  journal={Journal of Computational Physics},
  volume={162},
  number={2},
  pages={429--466},
  year={2000},
  publisher={Elsevier}
}

@article{pieraccini2007implicit,
  title={Implicit--explicit schemes for {BGK} kinetic equations},
  author={Pieraccini, Sandra and Puppo, Gabriella},
  journal={Journal of Scientific Computing},
  volume={32},
  pages={1--28},
  year={2007},
  publisher={Springer}
}

@article{boscheri2020high,
  title={High order central {WENO}-Implicit-Explicit {R}unge {K}utta schemes for the {BGK} model on general polygonal meshes},
  author={Boscheri, Walter and Dimarco, Giacomo},
  journal={Journal of Computational Physics},
  volume={422},
  pages={109766},
  year={2020},
  publisher={Elsevier}
}

@book{chapman1970mathematical,
  title={The mathematical theory of non-uniform gases: an account of the kinetic theory of viscosity, thermal conduction and diffusion in gases},
  author={Chapman, Sydney and Cowling, Thomas George},
  year={1970},
  publisher={Cambridge university press}
}

@article{itikawa1973effective,
  title={Effective collision frequency of electrons in gases},
  author={Itikawa, Yukikazu},
  journal={The physics of fluids},
  volume={16},
  number={6},
  pages={831--835},
  year={1973},
  publisher={AIP Publishing}
}

@article{yang2013kinetic,
  title={Kinetic numerical methods for solving the semiclassical Boltzmann-BGK equation},
  author={Yang, Jaw-Yen and Muljadi, Bagus Putra and Chen, Su-Yuan and Li, Zhi-Hui},
  journal={Computers \& Fluids},
  volume={85},
  pages={153--165},
  year={2013},
  publisher={Elsevier}
}

@article{aoki1990numerical,
  title={Numerical analysis of gas flows condensing on its plane condensed phase on the basis of kinetic theory},
  author={Aoki, Kazuo and Sone, Yoshio and Yamada, Tatsuo},
  journal={Physics of Fluids A: Fluid Dynamics},
  volume={2},
  number={10},
  pages={1867--1878},
  year={1990},
  publisher={American Institute of Physics}
}

@article{perthame1993weighted,
  title={Weighted ${L}^\infty$ bounds and uniqueness for the {B}oltzmann {B}GK model},
  author={Perthame, Beno{\^\i}t and Pulvirenti, Mario},
  journal={Archive for rational mechanics and analysis},
  volume={125},
  number={3},
  pages={289--295},
  year={1993},
  publisher={Springer}
}

@article{issautier1996convergence,
  title={Convergence of a weighted particle method for solving the {B}oltzmann ({BGK}) equation},
  author={Issautier, Didier},
  journal={SIAM journal on numerical analysis},
  volume={33},
  number={6},
  pages={2099--2119},
  year={1996},
  publisher={SIAM}
}

@article{yun2015classical,
  title={Classical solutions for the ellipsoidal {BGK} model with fixed collision frequency},
  author={Yun, Seok-Bae},
  journal={Journal of Differential Equations},
  volume={259},
  number={11},
  pages={6009--6037},
  year={2015},
  publisher={Elsevier}
}

@article{perthame1989global,
  title={Global existence to the {BGK} model of {B}oltzmann equation},
  author={Perthame, Benoˆ{\i}t},
  journal={Journal of Differential equations},
  volume={82},
  number={1},
  pages={191--205},
  year={1989},
  publisher={Elsevier}
}

@article{mischler1996uniqueness,
  title={Uniqueness for the {BGK}-equation in $\mathbb{R}^{N}$ and rate of convergence for a semi-discrete scheme},
  author={Mischler, St{\'e}phane},
  journal={Differential Integral Equations},
  volume={9},
  number={5},
  pages={1119--1138},
  year={1996}
}

@article{zhang2007lp,
  title={${L}^p$ solutions to the {C}auchy problem of the {BGK} equation},
  author={Zhang, Xianwen and Hu, Shigeng},
  journal={Journal of mathematical physics},
  volume={48},
  number={11},
  year={2007},
  publisher={AIP Publishing}
}

@article{zhang2010cauchy,
  title={On the {C}auchy problem of the {V}lasov-{P}oisson-{BGK} system: global existence of weak solutions},
  author={Zhang, Xianwen},
  journal={Journal of Statistical Physics},
  volume={141},
  number={3},
  pages={566--588},
  year={2010},
  publisher={Springer}
}

@article{wei2012cauchy,
  title={The {C}auchy problem for the {BGK} equation with an external force},
  author={Wei, Jinbo and Zhang, Xianwen},
  journal={Journal of mathematical analysis and applications},
  volume={391},
  number={1},
  pages={10--25},
  year={2012},
  publisher={Elsevier}
}

@article{yun2010cauchy,
  title={Cauchy problem for the {B}oltzmann-{BGK} model near a global Maxwellian},
  author={Yun, Seok-Bae},
  journal={Journal of mathematical physics},
  volume={51},
  number={12},
  year={2010},
  publisher={AIP Publishing}
}

@article{bellouquid2003global,
  title={Global existence and large-time behavior for {BGK} model for a gas with non-constant cross section},
  author={Bellouquid, A},
  journal={Transport theory and statistical physics},
  volume={32},
  number={2},
  year={2003},
  publisher={Taylor \& Francis}
}

@article{yun2019ellipsoidal,
  title={Ellipsoidal {BGK} model for polyatomic molecules near {M}axwellians: {A} dichotomy in the dissipation estimate},
  author={Yun, Seok-Bae},
  journal={Journal of Differential Equations},
  volume={266},
  number={9},
  pages={5566--5614},
  year={2019},
  publisher={Elsevier}
}

@article{brull2020stationary,
  title={Stationary Flows of the {ES}--{BGK} model with the correct Prandtl number},
  author={Brull, Stephane and Yun, Seok-Bae},
  journal={arXiv preprint arXiv:2012.08490},
  year={2020}
}

@article{bang2016stationary,
  title={Stationary solutions for the ellipsoidal {BGK} model in a slab},
  author={Bang, Jeaheang and Yun, Seok-Bae},
  journal={Journal of Differential Equations},
  volume={261},
  number={10},
  pages={5803--5828},
  year={2016},
  publisher={Elsevier}
}

@article{yang1995rarefied,
  title={Rarefied flow computations using nonlinear model {B}oltzmann equations},
  author={Yang, JY and Huang, JC},
  journal={Journal of Computational Physics},
  volume={120},
  number={2},
  pages={323--339},
  year={1995},
  publisher={Elsevier}
}

@article{russo2009semilagrangian,
  title={Semilagrangian schemes applied to moving boundary problems for the {BGK} model of rarefied gas dynamics},
  author={Russo, Giovanni and Filbet, Francis},
  journal={Kinetic and related models},
  volume={2},
  number={1},
  pages={231--250},
  year={2009}
}

@article{russo2012convergence,
  title={Convergence of a semi-Lagrangian scheme for the {BGK} model of the Boltzmann equation},
  author={Russo, Giovanni and Santagati, Pietro and Yun, Seok-Bae},
  journal={SIAM Journal on Numerical Analysis},
  volume={50},
  number={3},
  pages={1111--1135},
  year={2012},
  publisher={SIAM}
}

@article{russo2018convergence,
  title={Convergence of a semi-Lagrangian scheme for the ellipsoidal {BGK} model of the {B}oltzmann equation},
  author={Russo, Giovanni and Yun, Seok-Bae},
  journal={SIAM Journal on Numerical Analysis},
  volume={56},
  number={6},
  pages={3580--3610},
  year={2018},
  publisher={SIAM}
}

@article{duan2017global,
  title={Global existence for the ellipsoidal {BGK} model with initial large oscillations},
  author={Duan, Renjun and Wang, Yong and Yang, Tong},
  journal={SCIENTIA SINICA Mathematica},
  volume={47},
  number={10},
  pages={1143--1154},
  year={2017},
  publisher={Science China Press}
}

@article{bae2026large,
  title={Large amplitude problem of {BGK} model: Relaxation to quadratic nonlinearity},
  author={Bae, Gi-Chan and Ko, Gyounghun and Lee, Donghyun and Yun, Seok-Bae},
  journal={SIAM Journal on Mathematical Analysis},
  volume={58},
  number={2},
  pages={1530--1570},
  year={2026},
  publisher={SIAM}
}

@article{rapp2010equilibration,
  title={Equilibration rates and negative absolute temperatures for ultracold atoms in optical lattices},
  author={Rapp, Akos and Mandt, Stephan and Rosch, Achim},
  journal={Physical review letters},
  volume={105},
  number={22},
  pages={220405},
  year={2010},
  publisher={APS}
}

@unpublished{Yun2026,
  author = {Yun, Seok-Bae},
  title  = {Cauchy problem for the {B}oltzmann-{BGK} model with physical collision frequency},
  note   = {Preprint},
  year   = {2026}
}

\end{document}